\documentclass[preprint,12pt]{elsarticle}    

\usepackage{subcaption}
\usepackage{graphicx,epsfig}
\usepackage{enumerate}
\usepackage{amssymb}
\usepackage{amsthm}
\usepackage{amsmath}
\usepackage{centernot}
\usepackage{xcolor}
\usepackage{todonotes}
\usepackage{times}
\usepackage{mathrsfs}
\usepackage[utf8]{inputenc}

\usepackage{mathtools}
\usepackage{appendix}

\usepackage{fancyhdr}
\usepackage[ruled,vlined,resetcount,algosection]{algorithm2e}
\include{pythonlisting}

\SetCommentSty{mycommfont}
\usepackage{algpseudocode}
\usepackage{makecell}

\newtheorem{Remark}{Remark}
\numberwithin{Remark}{section}
\newtheorem{Lemma}{Lemma}
\numberwithin{Lemma}{section}
\newtheorem{Definition}{Definition}
\numberwithin{Definition}{section}
\newtheorem{Proposition}{Proposition}
\numberwithin{Proposition}{section}
\newtheorem{Corollary}{Corollary}
\numberwithin{Corollary}{section}
\newtheorem{Theorem}{Theorem}
\numberwithin{Theorem}{section}
\newtheorem{Algorithm}{Algorithm}
\numberwithin{Algorithm}{section}

\theoremstyle{definition}

\newtheorem{Problem}{Problem}
\numberwithin{Problem}{section}
\newsavebox\myboxA
\newsavebox\myboxB
\newlength\mylenA

\usepackage{amsmath}
\usepackage{wrapfig}
\usepackage{setspace}
\usepackage{booktabs}
\usepackage{lscape}
\usepackage{bm}
\usepackage{graphicx,epsfig}

\usepackage{xfrac}
\usepackage{url}
\usepackage{multirow}

\makeatletter
\newcommand{\thickhline}{%
    \noalign {\ifnum 0=`}\fi \hrule height 1pt
    \futurelet \reserved@a \@xhline
}
\numberwithin{equation}{section}

\begin{document}

\begin{frontmatter}

\title{\textbf{Tikhonov Regularized Iterative Methods for Nonlinear Problems }}
\author[iitbhu_math]{Avinash Dixit}
\ead{discover.avi92@gmail.com}
\author[radko2]{D. R. Sahu}
\ead{drsahudr@gmail.com}
\author[iitbhu_math]{Pankaj Gautam\corref{cor1}}
\ead{pankajg.rs.mat16@itbhu.ac.in}
\author[iitbhu_math]{T. Som}
\ead{tsom.apm@itbhu.ac.in}
\address[iitbhu_math]{Department of Mathematical Sciences,  Indian Institute of Technology (BHU) Varanasi \\ --221005, India}
\address[radko2]{Department of Mathematics, Banaras Hindu University, Varanasi-221005, India}
\cortext[cor1]{Corresponding author}

%
%
%

\begin{abstract}
 We consider the monotone inclusion problems in real Hilbert spaces. Proximal splitting algorithms are very popular technique to solve it and generally achieve weak convergence under mild assumptions. Researchers assume the strong conditions like strong convexity or strong monotonicity on the considered operators to prove strong convergence of the algorithms. Mann iteration method and normal S-iteration method are  popular methods to solve fixed point problems. We propose a new common fixed point algorithm based on normal S-iteration method {using Tikhonov regularization }to find common fixed point of nonexpansive operators and prove strong convergence of the generated sequence to the set of common fixed points without assuming strong convexity and strong monotonicity. Based on proposed fixed point  algorithm, we propose a forward-backward-type algorithm and a Douglas-Rachford algorithm in connection with Tikhonov regularization to find the solution of monotone inclusion problems. Further, we consider the complexly structured monotone inclusion problems which are very popular these days. We also propose a strongly convergent forward-backward-type primal-dual algorithm and a Douglas-Rachford-type primal-dual algorithm to solve the monotone inclusion problems. Finally, we conduct a numerical experiment to solve image deblurring problems. \\
\end{abstract}


\end{frontmatter}

\section{Introduction}

 Throughout the paper, $\mathcal{H}$ denotes a real Hilbert space with
	inner product $\langle \cdot ,\cdot \rangle $ and norm $\Vert \cdot \Vert $,
	respectively. Consider $\operatorname{T}: \mathcal{H} \to 2^{\mathcal{H}}$ is a set-valued monotone operator. The monotone inclusion problem is to find $x \in \mathcal{H}$ such that
\begin{equation}\label{P4e1.1}
	0 \in T(x).
\end{equation}
The monotone inclusion problem (\ref{P4e1.1}) plays important role in nonlinear analysis. Many problems arising in engineering, economics and physics can be framed as monotone inclusion problem  (see \cite{Barty2007,Bennar2007,Briceno2013,Combettes2005,Daubechies2004,Nemirovski2009}). Martinet \cite{Martinet1972} has proposed proximal point algorithm, which is very popular to solve monotone inclusion problem. The proximal point algorithm is given by
\begin{equation}\label{P4e1.2}
	x_{n+1}=J_{\lambda_n {T} }(x_n)\quad\forall n \in \mathbb{N},
\end{equation}
where $J_{\lambda_n {T} }= (Id+\lambda_n T)^{-1}$, $\lambda_n > 0$ is a regularization parameter and $x_1 \in \mathcal{H}.$ Rockafellar \cite{Rockafella1976,Rockafellar1976} has proved that proximal point algorithm converges weakly to solution set of inclusion problems in real Hilbert space framework. Further, he has introduced the inexact proximal point algorithm  as follows:
\begin{equation}\label{P4e1.3}
	x_{n+1}=J_{\lambda_n T} (x_n+\textcolor{red}{\epsilon_n}),\quad\forall n \in \mathbb{N},
\end{equation}
where $\{\epsilon_n\}$ is an error sequence in $\mathcal{H}$. The sequence $\{x_n\}$ also converges weakly to solution set of inclusion problem provided $\sum_{n=1}^{\infty} \| \epsilon_n\| < \infty.$ Guler \cite{Guler1991} has shown by an example that sequence generated by proximal point algorithm (\ref{P4e1.2}) converges weakly, but not strongly. It becomes a matter of interest for the research community to modify the proximal point algorithm to obtain the strong convergence. In such consequences, Tikhonov method has been proposed which generates as follows:
\begin{equation}
	x_{n+1}=J_{\lambda_n T} (x),
\end{equation}
where $x \in \mathcal{H}$ and $\lambda_n >0$ such that $\lambda_n \to \infty.$ Detailed study of Tikhonov regularization method can be found in \cite{Butnariu2008,Tikhonov1965,Tikhonov1977,Tikhonov1963,Xu2002}. Lehdili and Moudafi \cite{Lehdili1996} have combined the idea of proximal algorithm and Tikhonov regularization to find an algorithm converges strongly to solution of inclusion problem (\ref{P4e1.1}). They have solved the inclusion problem (\ref{P4e1.1}) by solving inclusion problem of fixed approximation of $T$, which is $T_n=T+\mu_n  Id$, i.e.,
\begin{equation*}
	\text{find } x\in \mathcal{H} \text{ such that } 0 \in T_n(x),
\end{equation*}
where $\mu_n$ is regularization parameter of $T$.
The proximal-Tikhonov algorithm is given by
$$x_{n+1}=J_{\lambda_{n}T_n}(x_{n}).$$
The Tikhonov regularization term $\mu_n Id$ has impelled the strong convergence to the algorithm. In absence of Tikhonov regularization term, proximal-Tikhonov algorithm becomes proximal algorithm which shows only weak convergence in most of the cases. The strong convergence of the algorithm can be obtained by using some other techniques also, some of them can be found in \cite{Bauschke2001,Haugazeau1968}.

Evaluation of resolvent is sometimes as hard as the original problem. This problem has been tried to resolve by splitting the operator in two operators, i.e., $T=A+B,$ whose resolvents are easy to compute. For $T=A+B,$ the monotone inclusion problem (\ref{P4e1.1}) becomes
\begin{equation}\label{P4e1.4}
	\text{find } x \in \mathcal{H} \text{ such that }0 \in (A+B)x,
\end{equation}
where $A : \mathcal{H} \to 2^\mathcal{H}$ is  maximally monotone operator  and $B$ is an operator. Problem (\ref{P4e1.1}) is also a generalization of the variational inequality problem:

\begin{equation*}
	\text{find } x_1 \in \mathcal{H} \text{ such that } (\exists x^* \in g(x_1)) (\forall x_2 \in \mathcal{H}) \langle x_1 -x_2, x^* \rangle \leq f(x_2) - f(x_1), 
\end{equation*}

where $f: \mathcal{H} \to \mathcal{H}$ is a proper, convex, lower semicontinuous and $g: \mathcal{H} \to 2^\mathcal{H}$ be a maximally monotone operator. 
The problem (\ref{P4e1.1}) serves as a blanket for various nonlinear problems viz. image denoising problem; clustering problem; wireless sensor network localization problem; matrix factorization problem; generalized Nash equilibrium problem and many more (see \cite{Bot2013,Bota2012,Bot2021,Gautam2021,Gur2020}).

Forward-backward splitting algorithm and Douglas-Rachford algorithm have been proposed to solve Problem (\ref{P4e1.4}). Forward-backward splitting method has been proposed by Lions and Mercier \cite{Lions1979}, Passty \cite{Passty1979}, which is given by
\begin{equation}
	x_{n+1}=(Id+\lambda_n A)^{-1}(Id-\lambda_n B)x_n,
\end{equation}
where $\lambda_n > 0$ and $B:\mathcal{H} \to \mathcal{H}$ is a cocoercive operator. Mercier \cite{Mercier1980} and Gabay \cite{Gabay1983} have studied the convergence behavior of forward-backward method when $A^{-1}$ is $\gamma$-strongly monotone with $\gamma >0.$ They have proved that forward-backward algorithm converges weakly to the point in the solution set provided $\lambda_n < 2 \gamma$, is constant. In addition, if  $A$ is strongly monotone, then $\{x_n\}$ shows strong convergence to the unique solution  of Problem (\ref{P4e1.4}). Chen and Rockafellar \cite{Chen1997} have also assumed the strong monotonicity of $A$ to prove the strong convergence of forward-backward method which depends on Lipschitz constant and modulus of strong monotonicity. Further, forward-backward method has been extensively studied, few of them can be found in (\cite{Chen1994,Chen1997,Mouallif1991,Moudafi1997}) and references therein. 

Douglas-Rachford method has been proposed to solve problem (\ref{P4e1.4}) when both $A$ and $B$ are set-valued. It has been originally proposed by Douglas and Rachford \cite{Douglas1956} to solve linear equations arising in heat-conduction problems. Lions and Mercier \cite{Lions1979} have extended the Douglas-Rachford algorithm to monotone operators. Douglas-Rachford algorithm is given as follows$\colon$
\begin{equation}
	x_{n+1}=R_B R_A x_n, \quad\forall n \in \mathbb{N},
\end{equation}
where $R_B$ and $R_A$ are reflected resolvent of operators $B$ and $A$, respectively. Lions and Mercier \cite{Lions1979} have proved that Douglas-Rachford algorithm converges weakly to a fixed point of operator $T$ which helps to obtain the solution of the Problem (\ref{P4e1.4}). Svaiter \cite{Svaiter2011} has supported the results of Lions and Mercier by proving the weak convergence of the shadow sequence to a solution. Further the analysis of the Douglas-Rachford algorithm can be found in (\cite{Artacho2013,Davis2015,Luke2020,Phan2016}).

Let $\mathcal{C}$ be a nonempty closed convex subset of $\mathcal{H}$ and $\mathcal{S}: \mathcal{C} \to \mathcal{C}$ be a nonexpansive operator. There are a number of iterative methods for finding fixed points of nonexpansive operators.  We recall some well known fixed point methods, which are given below$\colon$
\begin{enumerate}
	\item[$\bullet$] Mann iteration method \cite{Mann1953}:
	\begin{equation*}
		x_{n+1}=(1-\beta_n) x_n+ \beta_n \mathcal{S}(x_n),\quad\forall n \in \mathbb{N};
	\end{equation*}
	\item[$\bullet$] S-iteration method \cite{Agarwal2007}:
	\begin{equation*}
		x_{n+1}=(1-\alpha_n)\mathcal{S}(x_n) +\alpha_n \mathcal{S}[(1-\beta_n) x_n + \beta_n \mathcal{S} x_n],\quad\forall n \in \mathbb{N};
	\end{equation*}
	\item[$\bullet$] Normal S-iteration method \cite{Sahu2011}:
	\begin{equation*}
		x_{n+1}=\mathcal{S}[(1-\beta_n)x_n+\beta_n \mathcal{S} x_n],\quad\forall n \in \mathbb{N};
	\end{equation*}
\end{enumerate}
where $\alpha_n, \beta_n \in (0,1).$ The importance of these algorithms are not limited to solve fixed point problems, but these algorithms are also useful for solving inclusion problems of sum of a set-valued maximally monotone operator and a single-valued cocoercive operator, and inclusion problems of sum of two set-valued maximally monotone operators. The S-iteration methodology has been applied for solving
	various nonlinear problems, inclusion problems, optimization problems and
	image recovery problems. Recently, it has been demonstrated by Avinash et al. \cite{Dixit2019}  that the inertial normal S-iteration method has better performance compared to the inertial Mann iteration method. The S-iteration method and normal S-iteration method  are also useful  for finding common fixed points of nonexpansive operators. Since last few years, these properties of normal S-iteration make it popular among research community to find fixed point. Several research articles related to S-iteration and normal S-iteration can be found in \cite{Chang2013,Chang2014,Cholamjiak2015,Sahu2020,Sahuajit2020}. The weak convergence of the fixed point algorithms have reduced its applicability in infinite dimensional spaces. To achieve the strong convergence of algorithms one assumes stronger assumptions like strong monotonicity and strong convexity, which is difficult to achieve in many applications. This situation leaves a question to research community: can we find the strongly convergent algorithms without assuming these strong assumptions? The answer to this question is replied positively by Bo{\c{t}} et al. \cite{Bot2019}. They have modified the Mann algorithm  as follows:
\begin{equation}\label{P4e1.5}
	x_{n+1}=e_n x_n + \theta_n(\mathcal{S}(e_n x_n)-e_n x_n),
\end{equation}
where $e_n, \theta_n$ are positive real numbers. The strong convergence of algorithm (\ref{P4e1.5}) for nonexpansive operator $\mathcal{S}$  has been studied by Bo{\c{t}} et al. \cite{Bot2019}  when set of fixed points of $\mathcal{S}$ is nonempty and parameters $\theta_n$ and $e_n$ satisfy the following:
\begin{enumerate}
	\item [(i)] $0< e_{n} < 1$  for all $n \in \mathbb{N},$ $\lim\limits_{n \to \infty} e_{n} = 1$, $\sum_{n=1}^{\infty}(1-e_{n})= \infty$ and $\sum_{n=1}^{\infty}|e_{n}-e_{n-1}|< \infty;$
	\item  [(ii)] $0<\theta_{n}\leq 1$  for all $n \in \mathbb{N},$ $0<\liminf_{n\to \infty}\theta_{n},$  $\sum_{n=1}^{\infty} |\theta_{n}-\theta_{n-1}| < \infty.$
\end{enumerate}

  We consider the following more general problem:
\begin{Problem}\label{Prob1}
	Consider $\mathcal{T},\mathcal{S}:\mathcal{H} \to \mathcal{H}$ are nonexpansive operators. Find an element $x \in  \mathcal{H}$ such that $x \in \operatorname{Fix} (\mathcal{T})\cap \operatorname{Fix} (\mathcal{S})$.
\end{Problem}

The study on the common solutions of system of problems can be found in \cite{Gautam2020,Iemoto2009,Mainge2007}

	\begin{Remark} The algorithm (\ref{P4e1.5}) proposed by Bo{\c{t}} et al. \cite{Bot2019} can not directly apply to solve inclusion problem (\ref{P4e1.1}).
\end{Remark}

In this paper,  we introduce the normal S-iteration  method based fixed point algorithm   to find common fixed point of nonexpansive operators  $\mathcal{T}, \mathcal{S}:\mathcal{H} \to \mathcal{H}$, which converges strongly to common solutions of fixed point problem of operators $\mathcal{S}$ and $\mathcal{T}$. Based on the proposed fixed point algorithm, we develop a forward-backward algorithm  and a Douglas-Rachford algorithm containing Tikhonov regularization term to solve  the monotone inclusion problems.
In many cases, monotone inclusion problems are very complex, they contain mixture of linear and parallel sum monotone operators. Recently, many researchers have proposed primal-dual algorithms to precisely solve the considered complex monotone inclusion system \cite{Attouch1996,Bot2013,Brice2011,Combettes2012,BC2013}.  We have proposed a forward-backward type primal-dual algorithm and a Doughlas-Rachford type primal-dual algorithm having Tikhonov regularization term to find the common solution of the complexly structured monotone inclusion problems. The proposed algorithms have a special property that resolvents of all the operators are evaluated separately.\\

The paper is organized as follows: Next section recalls some important definitions and results in nonlinear analysis. In Section \ref{sec3}, we propose a normal S-iteration based Tikhonov regularized fixed point algorithm and study its convergence behavior. In Section \ref{sec4}, we propose a forward-backward-type algorithm and a forward-backward-type primal-dual algorithm to solve inclusion problem and complexly structured monotone inclusion problem, respectively. In Section \ref{sec5}, we propose Douglas-Rachford-type algorithms to solve monotone inclusion problems and complexly structured monotone inclusion problems of set-valued operators. In the last, we perform a numerical experiment to show the importance of proposed algorithms in solving image deblurring problems.



\section{Preliminaries}
This section devotes some important definitions and results from nonlinear analysis and operator theory.
Let $\mathbb{N}$ and $\mathbb{R}$ denote set of natural numbers and set of real numbers, respectively and  `$Id$' denotes identity operator. Consider the operator $T:\mathcal{H}\to 2^\mathcal{H}.$ Let $Gr(T)$  denote  the graph of $T$, $\operatorname{Zer}(T)$ denote  set of zeros of operator $T$ and $\operatorname{Fix} (T)$ denote  set of fixed points of $T$. The symbol $m$ is used to denote a strictly positive integer throughout the paper. The set of proper convex lower semicontinuous functions from $\mathcal{H}$ to $[-\infty,+\infty]$ is denoted by $\Gamma(\mathcal{H}).$ Let $f \in \Gamma(\mathcal{H}),$ then $argmin_{x\in \mathcal{H}}f(x)= \{x^*\in \mathcal{H}:f(x^*)\leq f(y), \forall y \in \mathcal{H}\}$ and $argmax_{x\in \mathcal{H}}f(x)=\{x^*\in \mathcal{H}: f(x^*) \geq f(y), \forall y \in \mathcal{H}\}.$ 
Let $A:\mathcal{H} \to 2^\mathcal{H}$ be an operator. Domain of $A$ is $\operatorname{dom}$ $(A)=\{x\in \mathcal{H}: Ax\neq \emptyset\}$. Range of $A$ is denoted by ran $(A)$ = $\cup_{x\in \mathcal{H}}$ $Ax$. $A$ is said to be monotone if
$$\langle x-y, u-v \rangle \geq 0, \ \forall (x,u), (y,v) \in Gr (A).$$
$A$ is said to be maximally monotone if there exists no monotone operator $B:\mathcal{H} \to 2^\mathcal{H}$ such that $Gr (B)$ properly contains $Gr (A).$ $A$ is strongly monotone with constant $\beta \in (0,\infty)$ if
\begin{equation*}
	\langle x-y,u-v \rangle \leq \beta \|x-y\|^2 \  \forall (x,u), (y,v) \in Gr(A).
\end{equation*}
 The resolvent of $A$ is defined by $J_A= (Id+A)^{-1} $ and the reflected resolvent of $A$ is $R_A= 2J_A -Id.$
Consider $f: \mathcal{H} \to [-\infty,\infty]$. The conjugate of $f$ is defined by $f^*:\mathcal{H} \to [-\infty,\infty],$ $$u \mapsto \sup_{x \in \mathcal{H}} \left( \langle x,u \rangle -f(x)\right). $$
Let $f:\mathcal{H} \to [-\infty,\infty]$ be a proper function. The subdifferential of $f$ is $\partial f:\mathcal{H} \to 2^\mathcal{H}$ is defined by
$$x \mapsto \{u \in \mathcal{H}| f(y) \geq f(x)+\langle y-x,u \rangle \  \forall y \in \mathcal{H}\}.$$
If $f \in \Gamma(\mathcal{H}),$ then $\partial f$ is maximally monotone. The resolvent of subdifferential of $f$ is $prox_f$,  where $prox_f:\mathcal{H}\to \mathcal{H}$ defined by
\begin{equation*}
	prox_f(x)=argmin_{y\in \mathcal{H}}\left\lbrace f(y)+\frac{1}{2}\|y-x\|^2\right\rbrace .
\end{equation*}

\begin{Definition}
	Let $\mathcal{C}$ be a nonempty subset of $\mathcal{H}$. Then:
	\item [(i)] interior of $\mathcal{C}$ is
	\begin{equation*}
		\text{ int } \mathcal{C}=\{x\in \mathcal{C}:(\exists \rho > 0) B(0;\rho) \subset \mathcal{C}-x\};
	\end{equation*}
	\item  [(ii)]strong relative interior of $\mathcal{C}$ is
	\begin{equation*}
		\text{sri } \mathcal{C} =\{x\in \mathcal{C}: cone(\mathcal{C}-x)=\overline{span}(\mathcal{C}-x) \};
	\end{equation*}
	\item  [(iii)]strong quasi-relative interior of $\mathcal{C}$ is
	\begin{equation*}
		\text{sqri } \mathcal{C} = \{x \in \mathcal{C}: \bigcup_{\rho > 0}\rho(\mathcal{C}-x) \text{is a closed linear subspace of space } \mathcal{H}\}.
	\end{equation*}
\end{Definition}
In case $\mathcal{H}$ is finite dimensional, sqri and sri are equivalent.

\begin{Definition}
	Let $C$ be a nonempty subset of $\mathcal{H},$ and $T:C\to \mathcal{H}$ be a nonexpansive operator. $T$ is said to be
	\begin{enumerate}
		\item[(i)] nonexpansive if
		\begin{equation*}
			\|Tx-Ty\|\leq \|x-y\| \ \forall x,y \in C;
		\end{equation*}
		\item[(ii)] firmly nonexpansive if
		\begin{equation*}
			\|Tx-Ty\|^2 +\|(Id-T)x-(Id-T)y\|^2 \leq \|x-y\|^2 \ \forall x,y \in C;
		\end{equation*}
		\item[(iii)] $\beta$-cocoercive ($\beta>0$) if
		\begin{equation*}
			\langle x-y,Tx-Ty \rangle \leq \beta \|Tx-Ty\|^2 \ \forall x, y \in C
		\end{equation*}
		\item[(iv)] $\alpha$-averaged for $\alpha \in (0,1)$ if there exists a nonexpansive operator $R:C\to \mathcal{H}$ such that $T=(1-\alpha)Id +\alpha R$.
	\end{enumerate}
\end{Definition}
An operator $T:\mathcal{H}\to 2^\mathcal{H}$ is strongly monotone with $\beta \in (0, \infty)$ implies $T^{-1}:\mathcal{H} \to \mathcal{H}$ is $\beta$-cocoerceive.
\begin{Definition}\cite{Bauschke2011}
	Let $\mathcal{C}$ be a nonempty subset of $\mathcal{H}.$ Then:
	
	\item [(i)] The indicator function $i_\mathcal{C}:\mathcal{H} \to[-\infty,+\infty]$ is defined by
	\begin{eqnarray}{
			i_\mathcal{C}(x)	=\left\{
			\begin{array}{lc@{}c@{}r}
				0, \ \ if \ x \in \mathcal{C}\\
				\infty \ \ otherwise
			\end{array}\right.}.
	\end{eqnarray}
	\item [(ii)] The projection of a point $x \in \mathcal{H}$ on $\mathcal{C}$ is defined by $\operatorname{proj}_C(x) =\left\lbrace u \in \mathcal{C}:u=argmin_{z \in \mathcal{C} }\|x-z\|\right\rbrace.$
	
	\item [(iii)]  Suppose $\mathcal{C}$ is convex, then normal cone to $\mathcal{C}$ at $x$ is defined by
	\begin{eqnarray}{
			N_\mathcal{C}(x)	=\left\{
			\begin{array}{lc@{}c@{}r}
				u \in \mathcal{H}:\sup \langle y-x,u \rangle \leq 0 ~\forall y \in \mathcal{C}, \ if \ x \in \mathcal{C}\\
				\emptyset, \ \ \ \ otherwise.
			\end{array}\right.}
	\end{eqnarray}
\end{Definition}

\begin{Definition}\cite[Proposition 4.32]{Bauschke2011}
	The parallel sum of two operators $T_1, T_2:\mathcal{H}\to 2^\mathcal{H}$ is $T_1 \Box T_2 :\mathcal{H} \to 2^\mathcal{H}$ defined by $T_1 \Box T_2  = (T_1 ^{-1}+T_2 ^{-1})^{-1}$.
\end{Definition}
The subdifferential of parallel sum of operators $T_1$ and $T_2$ is $\partial (T_1 \square T_2)=\partial T_1 \square \partial T_2$.
\begin{Remark}\label{P3R2.1}
	If $T_1 $ and $T_2 $ are monotone then the set of zeros of their sum $\operatorname{Zer}(T_1+T_2 )=J_{\gamma T_2} (\operatorname{Fix}(R_{\gamma T_1 } R_{\gamma T_2 }))$ $\forall \gamma >0$ and $R_{\gamma T_i}=2J_{\gamma T_i}-Id, i=1,2$.
\end{Remark}

\begin{Proposition}\cite{Bauschke2011}
	Consider $T_1,T_2 : \mathcal{H} \to \mathcal{H}$ be $\alpha_1,\alpha_2$-averaged operates, respectively. Then the averaged operator $T_1 \circ T_2$ is $\alpha= \frac{\alpha_1+\alpha_2-2\alpha_1\alpha_2}{1-\alpha_1\alpha_2}$-averaged.
\end{Proposition}

\begin{Lemma}\emph{\cite[Corollary 4.18]{Bauschke2011}}\label{P4L2.1}
	Let  $T:\mathcal{H} \to \mathcal{H}$ be a nonexpansive mapping. Let $\{u_n\}$ be a sequence in $\mathcal{H}$ and $u\in \mathcal{H}$ such that $u_n \rightharpoonup  u$ and $u_n-Tu_n \to 0$ as $n \to \infty$.
	Then $u \in \operatorname{Fix}(T)$.
\end{Lemma}
\begin{Lemma}\cite[Lemma 2.5]{Xu2002}\label{P4L2.2}
	Let $\{a_n\}$ be a sequence of nonnegative real numbers satisfying the inequality:
	$$a_{n+1}\leq (1-\theta_{n})a_n+\theta_n b_n+\epsilon_n \ \ \forall n \geq 0,$$
	 where
	\begin{enumerate}
		\item[(i)] $ 0\leq \theta_n \leq 1 $ for all $n \geq 0$ and $\sum_{n\geq 0} \theta_n = \infty;$
		\item[(ii)] $\limsup_{n\to \infty} b_n \leq 0$;
		\item[(iii)] $\epsilon_n \geq 0 \text{ for all } n \geq 0 \text{ and } \sum_{n\geq 0} \epsilon_n < \infty.$
	\end{enumerate}
	Then the sequence $\{a_n\}$ converges to $0$.
\end{Lemma}

\section{Tikhonov Regularized Strongly Convergent Fixed Point Algorithm}\label{sec3}
This section devotes  to investigate a computational theory for finding common fixed points of nonexpansive operators. We introduce a common fixed point algorithm such that sequence generated by the algorithm strongly converges to the set of common fixed points of mappings.
\begin{Algorithm}\label{P4A3.1}
	Let $\mathcal{S},\mathcal{T} :\mathcal{H} \to \mathcal{H}$ be nonexpansive mappings. Select $\{e_{n}\}$, $\{\theta_{n}\} \subset (0,1)$ and compute the $(n+1)^{th}$ iteration as follows:
	\begin{equation}\label{P4e3.1}
		y_{n+1}= \mathcal{S}[(1-\theta_n)e_{n}y_n+\theta_n \mathcal{T}(e_{n}y_n)] \ \ \ \  \textrm{for all} \ n\in\mathbb{N}.
	\end{equation}
\end{Algorithm}

We now study the convergence behavior of Algorithm \ref{P4A3.1} for finding the common fixed point of $\mathcal{S}$ and $\mathcal{T}$.  
\begin{Theorem}\label{P4T3.1}
	Let $\mathcal{S},\mathcal{T} :\mathcal{H} \to \mathcal{H}$ be nonexpansive mappings such that $\Omega:=\operatorname{Fix} (\mathcal{T})\cap \operatorname{Fix} (\mathcal{S})\neq \emptyset$. Let $\{y_n\}$ be a sequence in $\mathcal{H}$ defined by Algorithm \ref{P4A3.1}, where  $\{\theta_{n}\}$ and $\{e_{n}\}$ are real sequences satisfy the following conditions:
	\begin{enumerate}
		\item [\textbf{(i)}]\label{P4C1}
		$0< e_{n} < 1$  for all $n \in \mathbb{N},$ $\lim\limits_{n \to \infty} e_{n} = 1$, $\sum_{n=1}^{\infty}(1-e_{n})= \infty$ and $\sum_{n=1}^{\infty}|e_{n}-e_{n-1}|< \infty;$
		\item [\textbf{(ii)}]\label{P4C2}
		$0<\underline{\theta}\le \theta_{n}\leq\overline{\theta}< 1$   for all $n \in \mathbb{N},$  and  $\sum_{n=1}^{\infty} |\theta_{n}-\theta_{n-1}| < \infty.$
	\end{enumerate}
	Then the sequence $\{y_n\}$  converges strongly to $\operatorname{proj}_{\Omega}(0).$
\end{Theorem}
\begin{proof}
	In order to prove the convergence of the sequence $\{y_n\}$, we proceed with following steps:
	\begin{enumerate}
		\item[Step 1.]\label{C3.1} Sequence $\{y_n\}$ is bounded.\\
		
		Let  $y\in \Omega$. Since $\mathcal{S}$ and $\mathcal{T}$ are nonexpansive, we have following
		\begin{eqnarray}\label{P4e1}
			\|y_{n+1}-y\|&=& \|\mathcal{S}[(1-\theta_{n})e_{n}y_n + \theta_{n} \mathcal{T}(e_{n}y_n)]-y\| \nonumber\\ [6pt]
			&\leq &\|(1-\theta_{n})e_{n}y_n + \theta_{n} \mathcal{T}(e_{n}y_n)-y\| \nonumber\\[6pt]
			&\leq &(1-\theta_{n})\|e_{n}y_n -y\|+ \theta_{n}\| \mathcal{T}(e_{n}y_n)-y\| \nonumber\\[6pt]
			&\leq & \|e_{n}y_n-y\| \\[6pt]
			&=&\|e_{n}(y_n-y)-(1-e_{n})y\|\nonumber\\[6pt]
			&\leq & e_{n}\|(y_n-y)\|+(1-e_{n})\|y\|\nonumber\\[6pt]
			&\leq & \text{max}\{\|y_0-y\|,\|y\|\}\nonumber.
		\end{eqnarray}
		Thus, $\{y_n\}$ is bounded.\\
		
		\item[Step 2.]\label{C3.2} $\|y_{n+1}-y_n\| \to 0$ as $n \to \infty.$\\
		
		Using nonexpansivity of $\mathcal{S} \text{ and } \mathcal{T}$, we have
		\begin{eqnarray}
			\|y_{n+1}-y_n\|&=&\|\mathcal{S}[(1-\theta_{n})e_{n}y_n + \theta_{n} \mathcal{T}(e_{n}y_{n})]-\mathcal{S}[(1-\theta_{n-1})e_{n-1}y_{n-1} + \theta_{n-1} \mathcal{T}(e_{n-1}y_{n-1})]\| \nonumber\\[6pt]
			&\leq&\|(1-\theta_{n})e_{n}y_n + \theta_{n} \mathcal{T}(e_{n}y_{n})-(1-\theta_{n-1})e_{n-1}y_{n-1} - \theta_{n-1} \mathcal{T}(e_{n-1}y_{n-1})\| \nonumber\\[6pt]
			&=& \|(1-\theta_{n})e_{n}y_n-(1-\theta_{n-1})e_{n-1}y_{n-1}+\theta_{n} \mathcal{T}(e_{n}y_{n}) - \theta_{n-1} \mathcal{T}(e_{n-1}y_{n-1})\| \nonumber\\[6pt]
			&\leq&\|(1-\theta_{n})(e_{n}y_n-e_{n-1}y_{n-1})+(\theta_{n-1}-\theta_{n})e_{n-1}y_{n-1})\|\nonumber\\
			&+&\|\theta_{n} (\mathcal{T}(e_{n}y_{n})- \mathcal{T}(e_{n-1}y_{n-1}))+(\theta_{n}-\theta_{n-1})\mathcal{T}(e_{n-1}y_{n-1})\| \nonumber\\[6pt]
			&\leq&\|e_{n}y_n-e_{n-1}y_{n-1}\|+|\theta_{n}-\theta_{n-1}|\mathcal{C}_1 \nonumber\\[6pt]
			&=& \|e_{n}(y_n-y_{n-1})+(e_{n}-e_{n-1})y_{n-1}\|+|\theta_{n}-\theta_{n-1}|\mathcal{C}_1\nonumber\\[6pt]
			&\leq& e_{n}\|y_n-y_{n-1}\|+|e_{n}-e_{n-1}| \mathcal{C}_2 +|\theta_{n}-\theta_{n-1}|\mathcal{C}_1, \nonumber
		\end{eqnarray}
		for some $\mathcal{C}_1,$ $\mathcal{C}_2 >0$. By applying Lemma \ref{P4L2.2} with $a_n=\|y_n-y_{n-1}\|, b_n=0$, $ \epsilon_n=|e_{n}-e_{n-1}| \mathcal{C}_2 +|\theta_{n}-\theta_{n-1}|\mathcal{C}_1 \text{ and } \theta_{n}=1-e_{n}, \forall n \in \mathbb{N}$, we obtain that $\|y_{n+1}-y_n\| \to 0$.\\
		
		\item[Step 3.]\label{C3.3} $\|y_n-\mathcal{T}y_n\|$ and $\|y_n-\mathcal{S}y_n\| \to 0$ as $n \to \infty$.\\
		
		Let  $y\in \Omega$. Note\begin{eqnarray}
			\|y_{n+1}-y\|^2&=& \|\mathcal{S}[(1-\theta_{n})e_{n}y_n + \theta_{n} \mathcal{T}(e_{n}y_n)]-y\|^2 \nonumber\\ [6pt]
			&\leq &\|(1-\theta_{n})e_{n}y_n + \theta_{n} \mathcal{T}(e_{n}y_n)-y\|^2 \nonumber\\[6pt]
			&=&(1-\theta_{n})\|e_{n}y_n-y\|^2 + \theta_{n} \|\mathcal{T}(e_{n}y_n)-y\|^2 -\theta_{n}(1-\theta_{n}) \|e_{n}y_n-\mathcal{T}(e_{n}y_n)\|^2\nonumber\\ [6pt]
			&\leq&(1-\theta_{n})\|e_{n}y_n-y\|^2 + \theta_{n} \|e_{n}y_n-y\|^2 -\theta_{n}(1-\theta_{n}) \|e_{n}y_n-\mathcal{T}(e_{n}y_n)\|^2\nonumber\\[6pt]
			&=& \|e_{n}y_n-y\|^2 - \theta_{n}(1-\theta_{n})\|e_{n}y_n - \mathcal{T}(e_{n}y_n)\|^2,
		\end{eqnarray}
		which implies that

	\begin{align}\label{new1}
		\theta_{n}(1-\theta_{n})&\|e_{n}y_n - \mathcal{T}(e_{n}y_n)\|^2\nonumber\\ 
		&\leq \|e_{n}y_n-y\|^2 -\|y_{n+1}-y\|^2\nonumber\\
		&= \|e_n y_n -e_n y +e_n y -y\|^2-\|y_{n+1}-y\|^2\nonumber\\
		&= (1-e_n) \|y\|^2 + e_n \|y_n -y\|^2 -e_n (1-e_n) \|y_n\|^2 -\|y_{n+1}-y\|^2\nonumber\\
		&\leq (1-e_n) \|y\|^2 + e_n \|y_n -y\|^2 -e_n (1-e_n) \|y_n\|^2 -e_n\|y_{n+1}-y\|^2\nonumber\\
		&= e_n \{\|y_n -y\|^2 -\|y_{n+1}-y\|^2\}+(1-e_n) \|y\|^2 \nonumber\\
		&-e_n (1-e_n) \|y_n\|^2.
		\end{align}

	Since,
	\begin{equation*}
		\|y_n-y\|^2 \leq\|y_n-y_{n+1}\|^2 +\|y_{n+1}-y\|^2 +2 \|y_n-y_{n+1}\|\|y_{n+1}-y\|
	\end{equation*}
	which can be rewritten as
	\begin{equation*}
		\|y_n-y\|^2 -\|y_{n+1}-y\|^2  \leq \|y_n-y_{n+1}\|^2 + 2\|y_n-y_{n+1}\|\|y_{n+1}-y\|.
	\end{equation*}

	Thus (\ref{new1}) becomes
	\begin{eqnarray}
	\theta_{n}(1-\theta_{n})\|e_{n}y_n - \mathcal{T}(e_{n}y_n)\|^2 &\leq& e_n\{\|y_n-y_{n+1}\|^2 + 2\|y_n-y_{n+1}\|\|y_{n+1}-y\|\} \nonumber\\
	 &+&(1-e_n) \|y\|^2 -e_n (1-e_n) \|y_n\|^2. 
	\end{eqnarray}
Using Step 1, Step 2 and condition \textbf{(ii)} in Theorem \ref{P4A3.1}, we obtain $\theta_{n}(1-\theta_{n})\|e_{n}y_n - \mathcal{T}(e_{n}y_n)\|^2 \to 0.$\\

		Now 	\begin{eqnarray}
			\|y_n-\mathcal{T}y_n\| &=&\|y_n-e_{n}y_n+e_{n}y_n-\mathcal{T}(e_{n}y_n)+\mathcal{T}(e_{n}y_n)-\mathcal{T}y_n\| \nonumber\\
			&\leq& \|y_n-e_{n}y_n\|+\|e_{n}y_n-\mathcal{T}(e_{n}y_n)\|+\|\mathcal{T}(e_{n}y_n)-\mathcal{T}y_n\|  \nonumber\\
			&\leq& 2(1-e_{n}) \|y_n\|+\|e_{n}y_n-\mathcal{T}(e_{n}y_n)\| \to 0 \text{ as } n \to \infty. \nonumber
		\end{eqnarray}
		Observe that
		\begin{eqnarray}
			\|y_n-\mathcal{S}y_n\| &\leq& \|y_n-y_{n+1}\|+\|y_{n+1}-\mathcal{S}y_n\| \nonumber\\
			&=& \|y_n-y_{n+1}\|+\|\mathcal{S}[(1-\theta_n)e_{n}y_n + \theta_n \mathcal{T}(e_{n}y_n)]-\mathcal{S}y_n\| \nonumber\\
			&\leq&\|y_n-y_{n+1}\|+\|(1-\theta_n)e_{n}y_n + \theta_n \mathcal{T}(e_{n}y_n)-y_n\| \nonumber\\
			&\leq&\|y_n-y_{n+1}\|+(1-\theta_n)\|e_{n}y_n-y_n\| + \theta_n \|\mathcal{T}(e_{n}y_n)-y_n\| \nonumber\\
			&\leq&\|y_n-y_{n+1}\|+(1-\theta_n)(1-e_{n})\|y_n\|+\theta_n \|\mathcal{T}(e_{n}y_n)-\mathcal{T}y_n+\mathcal{T}y_n-y_n\| \nonumber\\
			&\leq&\|y_n-y_{n+1}\|+(1-\theta_n)(1-e_{n})\|y_n\|+\theta_n\|e_{n}y_n-y_n\|+\theta_{n}\|\mathcal{T}y_n-y_n\| \nonumber\\
			&=&\|y_n-y_{n+1}\|+(1-e_{n})\|y_n\|+\theta_n \|\mathcal{T}y_n-y_n\| \to 0 \text{ as } n \to \infty. \nonumber
		\end{eqnarray}
		Since $e_{n} \to 1$ and  $\|y_n-y_{n-1}\|$ and $\|y_n-\mathcal{T}y_n\| \to 0 \text{ as } n \to \infty,$
		we have $\|y_n-\mathcal{S}y_n\| \to 0.$
		
		\item[Step 4.] \label{C3.4}$\{y_n\}$ converges strongly to $\bar{y}=\operatorname{proj}_{\Omega}(0).$\\
		
		From (\ref{P4e1}), we set
		\begin{eqnarray}
			\|y_{n+1}-\bar{y}\| &\leq& \|\mathcal{S}[(1-\theta_n)e_{n}y_n + \theta_n \mathcal{T}(e_{n}y_n)]-\bar{y}\|
			\nonumber\\
			&=& \|(1-\theta_n)e_{n}y_n + \theta_n \mathcal{T}(e_{n}y_n)-\bar{y}\| \nonumber\\
			&=& \|(1-\theta_n)(e_{n}y_n-\bar{y}) + \theta_n \mathcal{T}((e_{n}y_n)-\bar{y})\| \nonumber\\
			&\leq& \|e_{n}y_n -\bar{y}\|.
		\end{eqnarray}
		Hence
		\begin{eqnarray}\label{P4e3.4}
			\|y_{n+1}-\bar{y}\|^2 &\leq& \|e_{n}y_n -\bar{y}\|^2\nonumber\\
			&\leq& \|e_{n}(y_n -\bar{y})-(1-e_{n})\bar{y}\|^2 \nonumber\\
			&\leq& e_{n}^2 \|y_n -\bar{y}\|^2+2 e_n(1-e_{n}) \langle -\bar{y}, y_n-\bar{y} \rangle+(1-e_{n})^2\|\bar{y}\|^2 \nonumber\\
			&\leq& e_{n} \|y_n -\bar{y}\|^2+2e_n(1-e_{n})  \langle -\bar{y}, y_n-\bar{y} \rangle+(1-e_{n})\|\bar{y}\|^2.
		\end{eqnarray}	
		Next we show that
		\begin{equation}\label{P4e3.5}
			\limsup_{n\to \infty} \langle -\bar{y}, y_n-\bar{y} \rangle\leq 0.
		\end{equation}
		Contrarily assume a real number $l$ and a subsequence $\{y_{n_j}\}$ of $\{y_n\}$ satisfying
		\begin{equation}
			\langle -\bar{y}, y_{n_j}-\bar{y}\rangle\geq l > 0 ~ \forall j\in \mathbb{N}.
		\end{equation}
		Since $\{y_n\}$ is bounded, there exists a subsequence $\{y_{n_j}\}$ which converges weakly to an element $y\in \mathcal{H}.$ Lemma \ref{P4L2.1} alongwith Step \ref{C3.2} implies that $y\in \Omega.$ By using variational characterazation of projection, we can easily derive
		\begin{equation}
			\lim_{j\to\infty}\langle -\bar{y}, y_{n_j}-\bar{y}\rangle = \langle -\bar{y}, y-\bar{y}  \rangle \leq 0,
		\end{equation}
		which is a contradiction. Thus, (\ref{P4e3.5}) holds and
		\begin{eqnarray}
			\limsup_{n\to \infty} \left( 2 e_n \langle -\bar{y}, y_n-\bar{y} \rangle+(1-e_{n})\|\bar{y}\|^2 \right)  \leq 0.
		\end{eqnarray}
		Consider $a_n=\|y_n-\bar{y}\|^2, b_n= 2 e_n \langle -\bar{y}, y_n-\bar{y} \rangle+(1-e_{n})\|\bar{y}\|^2, \epsilon_n= 0$ and $\theta_{n}= 1-e_{n}$ in (\ref{P4e3.4}) and apply Lemma \ref{P4L2.2}, we get the desired conclusion.
	\end{enumerate}
\end{proof}

\begin{Corollary}
	\label{c3.1}  Let $R_1, R_2:\mathcal{H}\to \mathcal{H}$ be $\alpha_1,\alpha_2
	$-averaged operators respectively, such that $\operatorname{Fix}{(R_1)}\cap \operatorname{Fix}(R_2) \neq
	\emptyset$. For $y_1\in \mathcal{H}$, let $\{y_n\}$ be sequence in $\mathcal{H}$ defined by
	\begin{equation}\label{key}
		y_{n+1}=R_2\{e_n y_n +\theta_{n}(R_1 (e_{n}y_n)-e_{n}y_n)\}\ \ \forall n\in
		\mathbb{N},
	\end{equation}
	where $\{\theta_{n}\}$ and $\{e_{n}\}$ are real sequences satisfy the
	condition (\textbf{i}) given in Theorem \ref{P4T3.1} and the conditions:
	$$0<\underline{\Theta }\leq \alpha _{1}\theta _{n}\leq
	\overline{\Theta }<1 \textrm{ for all } n\in \mathbb{N}\textrm{ and  }\sum_{n=1}^{\infty} |\theta
	_{n}-\theta _{n-1}|<\infty. $$
	Then the sequence $\{y_n\}$ converges strongly to $\operatorname{proj}_{\operatorname{Fix}{(R_1)}\cap
		\operatorname{Fix}(R_2)}(0).$
\end{Corollary}

\section{Tikhonov Regularized Forward-Backward-type Algorithms}\label{sec4}
In this section, we propose a forward-backward algorithm based on Algorithm \ref{P4A3.1} to simultaneously solve the monotone inclusion problems of the sum of two maximally monotone operators in which one is single-valued. Further, we also propose a forward-backward-type primal-dual algorithm based on Algorithm \ref{P4A3.1}  to solve a complexly structured monotone inclusion problem containing composition with linear operators and parallel-sum operators.

\subsection{Tikhonov Regularized Forward-Backward Algorithm}
Let $A_1,A_2: \mathcal{H} \to 2^\mathcal{H}$ be maximally monotone operators and $B_1,B_2: \mathcal{H} \to \mathcal{H}$ be $\alpha_1, \alpha_2$-cocoercive operators. We consider the monotone inclusion problem
\begin{equation}\label{P4Al3.1}
	\text{ find }x \in \mathcal{H} \text{ such that } 0 \in (A_1 +B_1)x \cap (A_2+B_2)x.
\end{equation}
We propose a forward-backward algorithm to solve the monotone inclusion problem (\ref{P4Al3.1}) such that generated sequence converges strongly to the solution set of the Problem (\ref{P4Al3.1}).
\begin{Theorem}\label{T4.1}
	Suppose $\operatorname{Zer}(A_1+B_1)\cap \operatorname{Zer}(A_2+B_2) \neq \emptyset$ and  $\gamma_1 \in (0,2\alpha_1)\text{ and }\gamma_2 \in (0,2\alpha_2).$ For $y_1 \in \mathcal{H}$, consider the forward-backward algorithm defined as follows:
	\begin{equation}\label{e4.1}
		y_{n+1}=J_{\gamma_2A_2}(Id-\gamma_2B_2)\left\lbrace (1-\theta_n)e_{n}y_n + \theta_n J_{\gamma_1 A_1}(e_n y_n-\gamma_1B_1(e_{n}y_n))\right\rbrace  \forall n \in \mathbb{N}.
	\end{equation}
	where $\{\theta_{n}\}$ and $\{e_{n}\}$ are real sequences satisfy the
	condition (\textbf{i}) given in Theorem \ref{P4T3.1} and the conditions:
	$$0<\underline{\Theta }\leq  \frac{2\alpha_1}{4\alpha_1-\gamma_1} \theta _{n}\leq
	\overline{\Theta }<1 \textrm{ for all } n\in \mathbb{N} \textrm{ and } \sum_{n=1}^{\infty}|\theta
	_{n}-\theta _{n-1}|<\infty .$$
	Then $\{y_n\}$ converges strongly to $\operatorname{proj}_{\operatorname{Zer}(A_1+B_1)\cap \operatorname{Zer}(A_2+B_2)}(0).$
\end{Theorem}
\begin{proof}
	Set $T_1=J_{\gamma_1A_1}(Id-\gamma_1B_1)$ and $T_2=J_{\gamma_2A_2}(Id-\gamma_2B_2),$ then algorithm (\ref{e4.1}) can be rewritten as:
	\begin{equation}
	y_{n+1}=T_2\{(1-\theta_{n})e_n y_n +\theta_{n}T_1 (e_{n}y_n)\} \ \forall \ n \in \mathbb{N}.
	\end{equation}
	Since $J_{\gamma_1A_1}$ is $\frac{1}{2}$-cocoercive \cite[Corollary 23.8]{Bauschke2011} and $Id-\gamma_1B_1$ is $\frac{\gamma_1}{2\alpha_1}$-averaged \cite[proposition 4.33]{Bauschke2011}, $T_1$ is $\frac{2\alpha_1}{4\alpha_1 -\gamma_1}$-averaged. Using the fact that $\operatorname{Zer}(A_i+B_i)= \operatorname{Fix}(T_i), i=1,2$ and assumption  $\operatorname{Zer}(A_1+B_1)\cap \operatorname{Zer}(A_2+B_2) \neq \emptyset$ \cite[Proposition 25.1]{Bauschke2011}, $\operatorname{Fix}(T_1) \cap \operatorname{Fix}(T_2) \neq \emptyset.$ Therefore, Theorem \ref{T4.1} follows from Corollary \ref{c3.1}.
\end{proof}

Further, we consider the following minimization problem and propose a proximal-point-type algorithm to solve it.
\begin{Problem}\label{P4P4.1}
	Consider strictly positive real numbers $\beta_1,\beta_2.$ Let $f_1, f_2 :\mathcal{H} \to \mathbb{R} \cup \{\infty\}$ be proper convex lower semicontinuous functions and $g_1,g_2: \mathcal{H} \to \mathbb{R}$ be convex and Frechet-differentiable functions with $\frac{1}{\beta_1},\frac{1}{\beta_2}$-Lipschitz continuous gradients, respectively. The problem is to find a point $y\in \mathcal{H}$ satisfying
	\begin{equation}\label{P4.1}
	y \in 	argmin_{x\in \mathcal{H}}\left\lbrace (f_1+g_1)(x)\right\rbrace  \bigcap  argmin_{x\in \mathcal{H}}\left\lbrace (f_2+g_2)(x)\right\rbrace .
	\end{equation}
\end{Problem}
\begin{Theorem}\label{c4.1}
	Consider the functions $f_1, f_2, g_1 \text{ and }g_2$ are as in Problem \ref{P4P4.1}. Let $argmin(f_1+g_1) \cap argmin(f_2+g_2) \neq \emptyset.$ For $\gamma_1 \in (0,2\beta_1] \text{ and } \gamma_2 \in (0,2\beta_2],$ consider an algorithm with initial point $y_1 \in \mathcal{H},$
	\begin{equation}\label{e4.2}
		y_{n+1}=prox_{\gamma_2f_2}(Id-\gamma_2\nabla g_2)\{(1-\theta_n)e_{n}y_n + \theta_n prox_{\gamma_1f_1}(e_n y_n-\gamma_1\nabla g_1(e_{n}y_n))\}  \  \ \forall n \in \mathbb{N},
	\end{equation}
	where $\theta_n \in (0,1] \text{ and }e_{n} \in (0,\frac{4\beta_1-\gamma_1}{2\beta_1})$ satisfy the condition (\textbf{i}) in Theorem \ref{P4T3.1} and the conditions:
	$$0< \underline{\Theta} \leq \frac{2\beta_1}{4\beta_1-\gamma_1}\theta_{n} \leq \overline{\Theta } <1 \textrm{ and } \sum_{n=1}^{\infty} |\theta_{n}-\theta_{n-1}| < \infty.$$
	Then $\{y_n\}$ converges strongly to minimal norm solution $y$ of Problem \ref{P4P4.1}.
\end{Theorem}
\begin{proof}
	Consider $A_1=\partial f_1, A_2=\partial f_2, B_1= \nabla g_1, B_1=\nabla g_2$. Since $$\operatorname{Zer}(\partial f_i +\nabla g_i)= argmin (f_i +g_i), \ i=1,2$$ and $\nabla g_1,\nabla g_2$ are $\beta_1,\beta_2$-cocoercive, respectively (Ballion-Hadded Theorem \cite[Corollary 16.18]{Bauschke2011}). Thus, by Theorem \ref{T4.1}, $\{y_n\}$ converges strongly to a point $y$ in $argmin(f_1+g_1) \cap argmin(f_2+g_2)$.
\end{proof}

\subsection{Tikhonov Regularized Forward-Backward-type Primal-Dual Algorithm}

\begin{Problem}\label{P6.1}
	Suppose $\Omega_1, \dots,\Omega_m$ are real Hilbert spaces. Consider the following operators:
	\begin{enumerate}
		$\bullet  A,B:\mathcal{H} \to 2^\mathcal{H}$ are maximally monotone operators,\\
		$\bullet C,D:\mathcal{H} \to \mathcal{H}$ are $\mu_1,\mu_2$-cocoercive operators, respectively, $\mu_1,\mu_2> 0,$\\
		$\bullet P_i,Q_i,R_i,  S_i:\Omega_i \to 2^{\Omega_i}$ are maximally monotone operators such that $Q_i$ is $\nu_i$-strongly monotone and  $S_i$ is $\delta_i$-strongly monotone, $\nu_i,\delta_i>0,$ $i=1, \dots, m$,\\
		$\bullet$ nonzero continuous linear operators $L_i: \mathcal{H} \to \Omega_i, i=1, \dots, m.$
	\end{enumerate}
	The primal inclusion problem is to find $\bar{y} \in \mathcal{H}$ satisfying
	\begin{eqnarray}
		& 0\in A \bar{y} +\sum_{i=1}^{m} L_i^{*}(P_i \square Q_i)(L_i \bar{y})+C\bar{y}\nonumber\\
		&\text{and}\nonumber\\
		&0\in B\bar{y} +\sum_{i=1}^{k} L_i ^*(R_i \square S_i)(L_i \bar{y})+D\bar{y} \nonumber
	\end{eqnarray}
	together with dual inclusion problem
	\begin{eqnarray}\label{e6.11}{find \ \ \bar{v}_1\in \varOmega_1, \dots, \bar{v}_m \in \varOmega_m \text{ such that } \exists y \in \mathcal{H} \text{ and }
			\left\{
			\begin{array}{lc@{}c@{}r}
				-\sum_{i=1}^{m}L_i^* \bar{v}_i \in Ay+Cy\\
				\bar{v}_i \in (P_i \square Q_i)(L_i y)\\
				\text{ and }\\
				-\sum_{i=1}^{m}L_i^* \bar{v}_i \in By+Dy\\
				\bar{v}_i \in (R_i \square S_i)(L_i y), \ \ \ i =1, \dots m.
			\end{array}\right.}
	\end{eqnarray}
\end{Problem}
A point $(\bar{y}, \bar{v}_1,\dots, \bar{v}_m) \in \mathcal{H}\times \varOmega_1 \times \cdots \times \varOmega_m$ be a primal-dual solution of Problem \ref{P6.1} if it satisfies the following:\\
\begin{eqnarray}{
		\left\{
		\begin{array}{lc@{}c@{}r}
			-\sum_{i=1}^{m}L_i^* \bar{v}_i \in A\bar{y}+C\bar{y},\\
			-\sum_{i=1}^{m}L_i^* \bar{v}_i \in B\bar{y}+D\bar{y}, \\
			\bar{v}_i \in (P_i \square Q_i)(L_i \bar{y}),\\
			\bar{v}_i \in (R_i \square S_i)(L_i \bar{y}) \ \ \ i =1,\dots, m.
		\end{array}\right.}
\end{eqnarray}

\begin{Theorem}\label{P4T6.1}
	Consider the operators as in Problem \ref{P6.1}. Assume
	\begin{equation}\label{e6.4}
		0\in ran \left( A+\sum_{i=1}^{m}L_i ^* \circ (P_i \square Q_i)\circ L_i +C\right)  \bigcap ran \left( B+\sum_{i=1}^{m}L_i ^* \circ (R_i \square S_i)\circ L_i +D\right).
	\end{equation}
	Let $\tau,\sigma_1,\dots, \sigma_m >0 $ such that
	\begin{equation*}
		2\rho\min\{\beta_1, \beta_2\} \geq 1,
	\end{equation*}
	where
	\begin{eqnarray*}
		\rho=\min\left\lbrace \frac{1}{\tau},\frac{1}{\sigma_1},\dots, \frac{1}{\sigma_m}\right\rbrace \left(1-\sqrt{\tau\sum_{i=1}^{m}\sigma_i\|L_i\|^2}\right)\\
		\beta_1=\min\{\mu_1,\nu_1,\dots,\nu_m\} \textrm{ and } \beta_2= \min\{\mu_2,\delta_1,\dots,\delta_m\}.
	\end{eqnarray*}
	
	Consider the algorithm with initial point $(y_1,v_{1,1},\dots, v_{m,1}) \in \mathcal{H} \times \varOmega_1 \times \cdots \times\varOmega_m$ and defined by

	\begin{algorithm}[H]
		\SetAlgoLined

		\KwIn{\begin{enumerate}
				\item initial points $(y_1,v_{1,1},\dots, v_{m,1}) \in \mathcal{H} \times \varOmega_1 \times \cdots \times\varOmega_m,$
				\item real numbers $\tau, \sigma_i > 0$, $i=1,2,...,m$ be such that $\tau\sum_{i=1}^{m} \sigma_i \|L_i\|^2 < 4, $
				\item $\theta_n \in (0,\frac{4\beta_1 \rho-1}{2\beta_1 \rho}], e_{n} \in (0,1)$ .
			\end{enumerate}	
		}
		\SetKwBlock{For}{For}{}
		\SetKwProg{}{}{}{}
		\For($n\in \mathbb{N}$; ){
			$p_n = J_{\tau A}\left[e_{n}y_n-\tau \left(e_{n}\sum_{i=1}^{m}L_i^* v_{i,n}+C(e_{n}y_n)\right)\right]$\\
			$u_n= e_{n}y_n+\theta_{n}(p_n-e_{n}y_n)$\\
			
			\For($i=1,\dots, m$; ){
				
				$q_{i,n}=J_{\sigma_iP_i^{-1}}\left[e_{n}v_{i,n}+\sigma_i(L_i(2p_n-e_{n}y_n)-Q_i^{-1}(e_{n}v_{i,n}))\right]$\\
				$u_{i,n}=e_{n}v_{i,n}+\theta_{n}(q_{i,n}-e_{n}v_{i,n})$\\
			}
			
			$y_{n+1}=J_{\tau B}\left[u_n-\tau \left(\sum_{i=1}^{m}L_i^* u_{i,n}+D(u_n)\right)\right]$\\
			$v_{i,n+1}=J_{\sigma_iR_i^{-1}}\left[u_{i,n}+\sigma_i(L_i(2y_{n+1}-u_n)-S_i^{-1}(u_{i,n}))\right]$

		}

		\KwOut{$(y_{n+1},v_{1,n+1},\dots,v_{m,n+1})$}		
		\caption{To optimize the complexly structured Problem \ref{P6.1} }\label{P4Al4.11}
	\end{algorithm}
	
	where sequences $\{\theta_{n}\}$ and $\{e_{n}\}$ satisfy the condition (\textbf{i}) and the conditions:
	$$0< \underline{\Theta} \leq \frac{2\beta_1 \rho}{4\beta_1 \rho-1}\theta_{n} \leq \overline{\Theta } <1   \textrm{ and } \sum_{n=1}^{\infty} |\theta_{n}-\theta_{n-1}| < \infty.$$
	Then there exists $(\bar{y},\bar{v}_1,\dots,\bar{v}_m)\in  \mathcal{H}\times \varOmega_1 \times \cdots \times \varOmega_m $ such that sequence $\{(y_n,v_{1,n},\dots,v_{m,n})\}$ converges strongly to  $(\bar{y},\bar{v}_1,\dots,\bar{v}_m)$ and satisfies the Problem \ref{P6.1}.
\end{Theorem}

\begin{proof}
	Consider the real Hilbert space $\mathcal{K }\equiv$ $\mathcal{H}\times \varOmega_1 \times \cdots \times \varOmega_m$ endowed with inner product
	$$\langle (x,u_1,\dots, u_m),(y,v_1,\dots,v_m)\rangle_{\mathcal{K}} =\langle x,y\rangle_\mathcal{H} +\sum_{i=1}^{m}\langle u_i,v_i \rangle_{\varOmega_i}$$
	and corresponding norm
	$$\|(x,u_1, \dots,u_m)\|_\mathcal{K}=\sqrt{\|x\|_\mathcal{H} ^2+\sum_{i=1}^{m}\|u_i\|^2 _{\varOmega_i} },\ \ \  \forall (x,u_1,\dots u_m),(y,v_1,\dots,v_m) \in \mathcal{K}.$$
	Further we consider following operators on real Hilbert space $\mathcal{K}$:
	
	\begin{enumerate}
		\item $\phi_1: \mathcal{K} \to 2^\mathcal{K}$, defined by $(x,u_1,\dots,u_m)\mapsto (Ax,P_1^{-1}u_1,\dots,P_m^{-1}u_m),$
		\item $\phi_2: \mathcal{K} \to 2^\mathcal{K}$, defined by $(x,u_1,\dots,u_m)\mapsto (Bx,R_1^{-1}u_1,\dots,R_m^{-1}u_m),$
		\item $\xi:\mathcal{K} \to \mathcal{K}$, defined by $(x,u_1,\dots,u_m)\mapsto \left(\sum_{i=1}^{m}L_i ^*u_i,-L_1 x,\dots, -L_m x\right),$
		\item $\psi_1:\mathcal{K} \to \mathcal{K}$, defined by $(x,u_1,\dots,u_m)\mapsto\left(Cx,Q_1^{-1}u_1,\dots, Q_m^{-1}u_m\right),$
		\item $\psi_2:\mathcal{K} \to \mathcal{K}$, defined by $(x,u_1,\dots,u_m)\mapsto\left(Dx,S_1^{-1}u_1,\dots, S_m^{-1}u_m\right).$
	\end{enumerate}
	These operators are maximally monotone as $A,B,P_i,R_i,Q_i, S_i, \ i=1,\dots,m$ are maximally monotone, $C,D \text{ are }\mu_1,\mu_2$-cocoercive, respectively and $\xi$ is skew-symmetric, i.e.,
	$\xi^*=-\xi$ (\cite[Proposition 20.22, 20.23 and Example 20.30]{Bauschke2011}).
	Now, define the continuous linear operator $\textbf{V}:\mathcal{K} \to \mathcal{K}$ by,
	$$(x,u_1,\dots,u_m) \to \left(\frac{x}{\tau}-\sum_{i=1}^{m}L_i^*u_i, \frac{u_1}{\sigma_1}-L_1x,\dots,\frac{u_m}{\sigma_m}-L_mx\right),$$
	which is self-adjoint and $\rho$-strongly positive, i.e., $\langle \textbf{x},\textbf{Vx}\rangle_{\mathcal{K}} \geq \rho \|\textbf{x}\|_{\mathcal{K}}^2 \ \ \forall \textbf{x}\in \mathcal{K}$ \cite{BC2013}. Therefore inverse of operator $\textbf{V}$ exists and satisfy $\|\textbf{V}^{-1}\|\leq \frac{1}{\rho}.$\\
	Now, consider the sequences
	\begin{eqnarray}{
			\forall n\in \mathbb{N}	\left\{
			\begin{array}{lc@{}c@{}r}
				\textbf{y}_n$=$(y_n,v_{1,n},\dots,v_{m,n}),\\ \textbf{u}_n$=$(u_n,u_{1,n},\dots,u_{m,n}),\\
				\textbf{x}_n =(p_n,q_{1,n},\dots,q_{m,n}).
			\end{array}\right.}
	\end{eqnarray}
	
	By taking into account the sequences $\{\textbf{y}_n \}$,$\{\textbf{x}_n\}$ and $\{\textbf{u}_n\}$ and operator \textbf{V}, Algorithm \ref{P4Al4.11} can be rewritten as
	\begin{eqnarray}\label{e6.22}{
			\forall n\in \mathbb{N}	\left\{
			\begin{array}{lc@{}c@{}r}
				e_{n}\text{\textbf{V}}(\text{\textbf{y}}_n)-\text{\textbf{V}}(\text{\textbf{x}}_n)-\text{\textbf{$\psi_1$}}(e_{n} \text{\textbf{y}}_n) \in (\phi_1+\xi)(\text{\textbf{x}}_n)\\
				\text{\textbf{u}}_n=e_{n}\text{\textbf{y}}_n+\theta_n(\text{\textbf{x}}_n-e_{n}\text{\textbf{y}}_n)\\
				\text{\textbf{V}\textbf{$\textbf{u}_n$}-\textbf{V}\textbf{y}}_{n+1}-\psi_2 \text{\textbf{$\textbf{u}_n$}} \in (\phi_2+\xi)\textbf{y}_{n+1}.
			\end{array}\right.}
	\end{eqnarray}
	On further analysing Algorithm $\ref{P4Al4.11}$, we get
	\begin{eqnarray}\label{e6.22}{
			\forall n\in \mathbb{N}	\left\{
			\begin{array}{lc@{}c@{}r}
				\textbf{x}_n=J_{\textbf{A}_1}(e_{n}\textbf{y}_n-\textbf{B}_1(e_{n}\textbf{y}_n)) \\
				\textbf{u}_n=e_{n}\textbf{y}_n+\theta_n(\textbf{x}_n-e_{n}\textbf{y}_n)\\
				\textbf{y}_{n+1}=J_{\textbf{A}_2}(\textbf{u}_n-\textbf{B}_2 \textbf{u}_n),
			\end{array}\right.}
	\end{eqnarray}
	where $\textbf{A}_1=\textbf{V}^{-1}(\phi_1 +\xi)$ ,\ $\textbf{B}_1=\textbf{V}^{-1}\psi_1$,\ $\textbf{A}_2=\textbf{V}^{-1}(\phi_2 +\xi)$ and $\textbf{B}_2=\textbf{V}^{-1}\psi_2.$ The Algorithm (\ref{e6.22}) is in the form of Algorithm (\ref{P4Al3.1}) for $\gamma=1$ and $A_i=\textbf{A}_i$ and $B_i=\textbf{B}_i, \ i=1,2$.
	Now, we define the real Hilbert space $\mathcal{K}_{\textbf{V}} \equiv \mathcal{H}\times \varOmega_1 \times \cdots \times \varOmega_m$ endowed with inner product $\langle \textbf{x},\textbf{y}\rangle_{\mathcal{K}_{\textbf{V}}}=\langle \textbf{x},{\textbf{V}}\textbf{y}\rangle_{\mathcal{K}}$ and corresponding norm is given by, $\|\textbf{x}\|_{\mathcal{K}_{\textbf{V}}}=\sqrt{\langle \textbf{x},{\textbf{V}}\textbf{x}\rangle_{\mathcal{K}}} \ \ \forall \textbf{x},\textbf{y} \in \mathcal{K}_{\textbf{V}}.$ \\
	In view of real Hilbert space $\mathcal{K}_{\textbf{V}}$ and Algorithm \ref{P4Al4.11}, we observe the following:
	\begin{enumerate}
		\item[(i)] since dom$(\xi)=\mathcal{K}$, $\phi_i +\xi$ are maximally monotone on $\mathcal{K}$ and thus maximal monotonocity of $\textbf{A}_i$ and $\textbf{B}_i$ on $\mathcal{K}_{\textbf{V}}$ are  followed, for $i=1,2$ \cite{BC2013}.
		\item[(ii)] $\textbf{B}_i$  are $\beta_i \rho$-cocoercive on $\mathcal{K}_{\textbf{V}}$ as $\psi_i$ are $\beta_i$-cocoercive in $\mathcal{K}$ \cite[Page 672]{BC2013}, for $i=1,2$.
		\item[(iii)] $\operatorname{Zer}(\textbf{A}_i+\textbf{B}_i$)=$\operatorname{Zer}({\textbf{V}^{-1}}(\phi_i+\xi+\psi_i))$=$\operatorname{Zer}(\phi_i+\xi+\psi_i),\ i=1,2$ and from condition $(\ref{e6.4})$, we can easily obtain that $\operatorname{Zer}(\textbf{A}_1+\textbf{B}_1)\cap \operatorname{Zer}(\textbf{A}_2+\textbf{B}_2) \neq \emptyset.$
	\end{enumerate}
	\par 	Since \textbf{V} is self-adjoint and $\rho$-strongly positive, weak convergence and strong convergence of sequences are same in both Hilbert spaces $\mathcal{K}$ and $\mathcal{K}_{\textbf{V}}$. Operators $\textbf{A}_i, \textbf{B}_i, \ i=1,2$ and sequences $\{\theta_{n}\},\{e_{n}\}$ satisfy the assumptions in Theorem $\ref{T4.1}$, therefore, according to Theorem \ref{T4.1}, $\{\textbf{y}_n\}$  converges strongly to $(\bar{y},\bar{v}_1,\dots,\bar{v}_m)\in$ $\operatorname{proj}_{\operatorname{Zer}(\textbf{A}_1+\textbf{B}_1)\cap \operatorname{Zer}(\textbf{A}_2+\textbf{B}_2)}(0,\dots,0)$ in the space ${\mathcal{K}}_\textbf{V}$ as $n \to \infty.$ Thus, we obtain the conclusion as $(\bar{y},\bar{v}_1,\dots,\bar{v}_m)\in \operatorname{Zer}(\phi_1+\xi+\psi_1)\cap \operatorname{Zer}(\phi_2+\xi+\psi_2),$  also satisfy primal-dual Problem \ref{P6.1}.
	
\end{proof}
Next, we define a complexly structured convex optimization problem and their Fenchel duals. Further, we propose an algorithm to solve the considered problem and study the convergence property of algorithm to find simultaneously the common solutions of optimization problems and common solutions of their Fenchel duals. Let $m$ is a positive integer. We have considered the following problem:
\begin{Problem}\label{P6.2}
	Let $f_1,f_2\in \varGamma(\mathcal{H})$ and $h_1,h_2$ be convex differentiable function with $\mu_1^{-1},\mu_2^{-1}$- Lipschitz continuous gradients, for some $\mu_1,\mu_2 >0.$ Let $\varOmega_i$ be real Hilbert spaces and $g_i, l_i,s_i,t_i \in \varGamma(\varOmega_i)$ such that $l_i,t_i$ are $\nu_i,\delta_i$-strongly convex, respectively, and $L_i:\mathcal{H} \to \varOmega_i$ be non-zero linear continuous operator $\forall i=1,2,\dots,m$. The optimization problem under consideration is to find $y \in \mathcal{H}$ such that
	\begin{eqnarray}
	y \in 	argmin_{x\in \mathcal{H}}\left\lbrace f_1(x)+\sum_{i=1}^{m}(g_i\square l_i)(L_i x)+h_1(x)\right\rbrace \nonumber\\ \bigcap argmin_{x\in \mathcal{H}}\left\lbrace f_2(x)+\sum_{i=1}^{m}(s_i\square t_i)(L_i x)+h_2(x)\right\rbrace
	\end{eqnarray}
	with its Fenchel-dual problem is to find $(v^*_1,\dots,v^*_m) \in \varOmega_1 \times\cdots \times \varOmega_m $ such that
	\begin{eqnarray}
	(v^*_1,\dots, v^*_m)\in 	argmax_{(v_1,\dots,v_m) \in \varOmega_1 \times \cdots \times \varOmega_m}\left\lbrace -(f_1^{*} \square h_1^{*})\left( -\sum_{i=1}^{m}L_i^* v_i\right) -\sum_{i=1}^{m}\left( g_i^*(v_i)+l_i^*(v_i)\right) \right\rbrace \nonumber\\
		\bigcap argmax_{(v_1,\dots,v_m) \in \varOmega_1 \times \cdots \times \varOmega_m}\left\lbrace -(f_2^{*} \square h_2^{*})\left( -\sum_{i=1}^{m}L_i^* v_i\right) -\sum_{i=1}^{m}\left( s_i^*(v_i)+t_i^*(v_i)\right) \right\rbrace.
	\end{eqnarray}
\end{Problem}
In following corollary, we propose an algorithm and study its convergence behaviour. The point of convergence will be the solution of Problem \ref{P6.2}.

\begin{Corollary}
	Assume  in Problem \ref{P6.2}
	\begin{equation}\label{e6.5}
		0\in ran \left( \partial f_1+\sum_{i=1}^{m}L_i ^* \circ (\partial g_i \square \partial l_i)\circ L_i +\nabla h_1 \right)  \bigcap ran \left( \partial f_2+\sum_{i=1}^{m}L_i ^* \circ (\partial s_i \square \partial t_i)\circ L_i +\nabla h_2\right).
	\end{equation}
	Consider $\tau, \sigma_i >0,  \ i=1,2,\dots,m$ such that
	\begin{equation*}
		2\rho\min\{\beta_1,\beta_2\} \geq 1.
	\end{equation*}
	where
	\begin{eqnarray*}
		\rho= \min\{\tau^{-1}, \sigma_1^{-1}, \dots, \sigma_m^{-1}\}\left(1-\sqrt{\tau\sum_{i=1}^{m}\sigma_i\|L_i\|^2}\right),\\
		\beta_1=\min\{\mu_1,\nu_1,\dots,\nu_m\}  \textrm{ and }   \beta_2= \min\{\mu_2,\delta_1,\dots,\delta_m\}.
	\end{eqnarray*}
	
	Consider the iterative scheme with intial point $(x_1,v_{1,1},\dots, v_{m,1}) \in \mathcal{H} \times \varOmega_1 \times \cdots \times\varOmega_m$ and defined by
	\begin{eqnarray}\label{e6.6}{
			\forall n\in \mathbb{N}	\left\{
			\begin{array}{lc@{}c@{}r}
				p_n = prox_{\tau f_1}\left[e_{n}x_n-\tau \left(e_{n}\sum_{i=1}^{m}L_i^* v_{i,n}+\nabla h_1(e_{n}x_n)\right)\right]\\
				u_n= e_{n}x_n+\theta_{n}(p_n-e_{n}x_n)\\
				For \ i=1,2,\dots, m\\
				
				q_{i,n}=prox_{\sigma_ig_i^{*}}\left[e_{n}v_{i,n}+\sigma_i(L_i(2p_n-e_{n}x_n)-\nabla l_i ^{*}(e_{n}v_{i,n}))\right]\\
				u_{i,n}=e_{n}v_{i,n}+\theta_{n}(q_{i,n}-e_{n}v_{i,n})\\
				x_{n+1}=prox_{\tau f_2}\left[u_n-\tau \left(\sum_{i=1}^{m}L_i^* u_{i,n}+\nabla h_2(u_n)\right)\right]\\
				v_{i,n+1}=prox_{\sigma_i s_i^{*}}\left[u_{i,n} +\sigma_i(L_i(2x_{n+1}-u_n)-\nabla t_i^{*}(u_{i,n}))\right],
			\end{array}\right.}
	\end{eqnarray}
	where sequences $\{\theta_{n}\}$ and $\{e_n\}$ satisfy the condition (\textbf{i}) in Theorem \ref{P4T3.1} and the conditions:
	$$0< \underline{\Theta} \leq \frac{2\beta_1 \rho}{4\beta_1 \rho-1} \theta_{n} \leq \overline{\Theta } <1 \textrm{ and } \sum_{n=1}^{\infty} |\theta_{n}-\theta_{n-1}| < \infty.$$
	
	Then, there exists $(\bar{x},\bar{v}_1,\dots,\bar{v}_m)\in  \mathcal{H}\times \varOmega_1 \times \cdots \times \varOmega_m $ such that sequence $\{(x_n,v_{1,n},\dots,v_{m,n})\}$ converges strongly to  $(\bar{x},\bar{v}_1,\dots,\bar{v}_m)$ as $n \to \infty$ and $(\bar{x},\bar{v}_1,\dots,\bar{v}_m)$ satisfies Problem \ref{P6.2}.
\end{Corollary}

\section{Tikhonov Regularized Douglas-Rachford-type Algorithms}\label{sec5}
In this section, using Algorithm \ref{P4A3.1} we propose a new Douglas-Rachford algorithm to solve monotone inclusion problem of sum of two maximally monotone operators. Further using proposed Douglas-Rachford algorithm, we propose a Douglas-Rachford-type primal-dual algorithm to solve complexly structured monotone inclusion problem containing linearly composite and parallel-sum  operators.

\subsection{Tikhonov Regularized Douglas-Rachford Algorithm}
Let $A, B:\mathcal{H} \to 2^\mathcal{H}$ be maximally monotone operators. In this section, we consider the following monotone inclusion problem:
\begin{equation}
	\text{ find } x\in \mathcal{H} \text{ such that } 0 \in (A+B)x .
\end{equation}
We propose a new  Douglas-Rachford algorithm based on Algorithm \ref{P4A3.1} such that generated sequence converges strongly to a point in the solution set.

\begin{Theorem}\label{T5.1}
	Consider $x_1 \in \mathcal{H}$ and $\gamma >0,$ then algorithm is given by:
	\begin{eqnarray}\label{e5.1}{
			\forall n\in \mathbb{N}\left\{
			\begin{array}{lc@{}c@{}r}
				y_n=J_{\gamma B}(e_{n}x_n)\\
				z_n=J_{\gamma A}(2y_n-e_{n}x_n)\\
				u_n=e_{n}x_n +\theta_{n}(z_n-y_n)\\
				x_{n+1}=(2J_{\gamma A}-Id)(2J_{\gamma B}-Id)u_n. \ \
			\end{array}\right.}
	\end{eqnarray}
	Let $\operatorname{Zer}(A +B)  \neq \emptyset$ and sequences $ e_n \in (0,1)$ and $\theta_{n} \in (0,2]$ satisfy the condition (\textbf{i}) in Theorem \ref{P4T3.1} and the conditions:
	
	$$0< \underline{\Theta} \leq \frac{\theta_{n}}{2} \leq \overline{\Theta } < 1 \textrm{ and } \sum_{n=1}^{\infty} |\theta_{n}-\theta_{n-1}| < \infty.$$
	
	Then the following statements are true:
	\begin{enumerate}
		\item[(i)] $\{x_n\}$ converges strongly to $\bar{x}=\operatorname{proj}_{\operatorname{Fix}R_{\gamma A} R_{\gamma B}}(0)$.
		\item[(ii)] \label{P4i2} $\{y_n\}$ and $\{z_n\}$ converges strongly to $J_{\gamma B}(\bar{x})\in \operatorname{Zer}(A+B)$.
	\end{enumerate}
\end{Theorem}
\begin{proof}
	Consider the operator $T\equiv R_{\gamma A} R_{\gamma B}:\mathcal{H} \to 2^\mathcal{H}$. From the definition of reflected resolvent, and definition of operator $T$, algorithm (\ref{e5.1}) can be rewritten as
	\begin{eqnarray}
		x_{n+1}&=&R_{\gamma A}R_{\gamma B}\{e_{n}x_n+\frac{\theta_{n}}{2}(R_{\gamma A}R_{\gamma B})(e_{n}x_n)-e_{n}x_n\}\nonumber\\
		&=& T\{e_{n}x_n +\frac{\theta_{n}}{2}(T(e_{n}x_n)-e_{n}x_n)\}.
	\end{eqnarray}
	Since resolvent operator is nonexpansive \cite[Corollary 23.10(ii)]{Bauschke2011}, $T$ is nonexpansive. Suppose $x^* \in \operatorname{Zer}(A+B)$ and results from \cite[Proposition 25.1(ii)]{Bauschke2011}, we have $\operatorname{Zer}(A+B)=J_{\gamma B}(\operatorname{Fix}(T)),$ which collectively implies that $\operatorname{Fix} {(T )} \neq \emptyset.$ Applying Theorem \ref{P4T3.1} with $A_1=A_2=A,B_1=B_2=B$, we conclued that $\{x_n\}$ converges strongly to $\bar{x}=\operatorname{proj}_{\operatorname{Fix} (T)}(0)$ as $n \to \infty.$ \\
	The continuity of resolvent operator forces the sequence $\{y_n\} $ to converge strongly to $J_{\gamma B}\bar{x} \in \operatorname{Zer}(A+B).$ Finally, since $z_n-y_n=\frac{1}{2}(T(e_{n}x_n)-e_{n}x_n),$ which converges strongly to 0, concludes \textbf{(ii)}. \\
\end{proof}

\begin{Problem}\label{P4P4.11}
	Let $f,g:\mathcal{H} \to \mathbb{R} \cup \{\infty\}$ are convex proper and lower semicontinuous functions. Consider the minimization problem
	\begin{equation}\label{P4e4.11}
		\min_{x\in \mathcal{H}} f(x) +g(x).
	\end{equation}
\end{Problem}
Using Karush-Kuhn-Tucker condition, (\ref{P4e4.11}) is equivalent to solve the inclusion problem
\begin{equation}
	\text{find} {\ x \in \mathcal{H}} \  0 \in \partial f(x) +\partial g(x).
\end{equation}

In order to solve such type of problem, we propose an iterative scheme and study its convergence behavior which can be summarized in following corollary.
\begin{Corollary}
	Let $f,g$ be as in Problem \ref{P4P4.11} with $argmin_{x \in \mathcal{H}}\{f (x)+g (x)\} \neq \emptyset$ and $0 \in sqri (dom \ f-dom \ g).$ Consider the following iterative scheme with $x_1 \in \mathcal{H}$:
	\begin{eqnarray}\label{e5.11}{
			\forall n\in \mathbb{N}\left\{
			\begin{array}{lc@{}c@{}r}
				y_n=prox_{\gamma g}(e_{n}x_n)\\
				z_n=prox_{\gamma f}(2y_n-e_{n}x_n)\\
				u_n=e_{n}x_n +\theta_{n}(z_n-y_n)\\
				x_{n+1}=(2prox_{\gamma f}-Id)(2prox_{\gamma g}-Id)u_n, \ \
				
			\end{array}\right.}
	\end{eqnarray}
	where $\gamma >0$
	and sequences $\{\theta_{n}\} \text{ and } \{e_{n}\}$ satisfy the condition (\textbf{i}) in Theorem \ref{P4T3.1} and the conditions:
	
	$$0< \underline{\Theta} \leq \frac{\theta_{n}}{2} \leq \overline{\Theta } < 1 \textrm{ and } \sum_{n=1}^{\infty} |\theta_{n}-\theta_{n-1}| < \infty.$$
	Then we have the following:
	\begin{enumerate}
		\item[(i)] $\{x_n\}$ converges strongly to $\bar{x}= \operatorname{proj}_{\operatorname{Fix}(T)}(0)$ where $T=(2prox_{\gamma f}-Id)(2prox_{\gamma g}-Id)$.
		\item[(ii)] $\{y_n\}$ and $\{z_n\}$ converge strongly to $prox_{\gamma g}(\bar{x})\in argmin_{x \in \mathcal{H}}\{f (x)+g (x)\} $.
	\end{enumerate}
	\begin{proof}
		Since $argmin_{x \in \mathcal{H}}\{f (x)+g (x)\} \neq \emptyset$ and $0 \in sqri (dom \ f-dom \ g)$ (\cite[Proposition 7.2]{Bauschke2011}) ensures that $\operatorname{Zer}(A+B)=argmin_{x \in \mathcal{H}}\{f (x)+g (x)\}.$ The results can be obtained by choosing $A =\partial f, B= \partial g $ in Theorem \ref{T5.1}.
	\end{proof}
\end{Corollary}

\subsection{Tikhonov Regularized Douglas-Rachford-type Primal-Dual Algorithm}
In this section, we propose Douglas-Rachford-type primal-dual algorithms to solve the complex-structured monotone inclusion problem having mixtures of composition of linear operators and parallel-sum operators. We consider the monotone inclusion problem is as follows:
\begin{Problem}\label{P4P6.2}
	Let $A:\mathcal{H} \to 2^\mathcal{H}$ be a maximally monotone operator. Consider for each $i=1, \dots,m$, $\varOmega_i$ is a real Hilbert space, $P_i, Q_i: \varOmega_i \to 2^{\varOmega_i}$ are maximally monotone operators and $L_i:\mathcal{H} \to \varOmega_i$ are nonzero linear continuous operator.
	The primal inclusion problem is to find $\bar{x} \in \mathcal{H}$ satisfying
	\begin{eqnarray}
		0\in A \bar{x} +\sum_{i=1}^{m} L_i^{*}(P_i \square Q_i)(L_i \bar{x})\nonumber
	\end{eqnarray}
	together with dual inclusion problem
	\begin{eqnarray}\label{e6.1}{find \ \ \bar{v}_1\in \varOmega_1, \dots, \bar{v}_m \in \varOmega_m \text{ such that } (\exists x \in \mathcal{H})
			\left\{
			\begin{array}{lc@{}c@{}r}
				-\sum_{i=1}^{m}L_i^* \bar{v}_i \in Ax\\
				\bar{v}_i \in (P_i \square Q_i)(L_i x), \ i=1,\dots,m.
			\end{array}\right.}
	\end{eqnarray}
\end{Problem}

Here, operators $P_i ,Q_i, i=1,\dots,m $ are not cocoercive, thus to solve the Problem \ref{P4P6.2}, we have to evaluate the resolvent of each operator, which makes the Douglas-Rachford algorithm based primal-dual algorithm is more appropriate to solve the problem.

\begin{Theorem}\label{P4T5.11}
	In addition to assumption in Problem \ref{P4P6.2}, we assume that
	
	\begin{equation}\label{P4P6.10}
		0 \in ran \left( A+\sum_{i=1}^{m}L_i ^* \circ(P_i \square Q_i) \circ L_i \right).
	\end{equation}
	Consider the strictly positive integers $\tau, \sigma_i, i=1,\dots,m$ satisfying
	\begin{equation}
		\tau \sum_{i=1}^{m} \sigma_i\|L_i\|^2 < 4.
	\end{equation}
	
	Consider the initial point $(x_1,v_{1,1},\dots,v_{m,1})\in \mathcal{H}\times \varOmega_i \times \cdots \times \varOmega_m.$ The primal- dual algorithm to solve Problem \ref{P4P6.2} is given by

	\begin{algorithm}[H]
		\SetAlgoLined
		
		\KwIn{\begin{enumerate}
				\item initial points $(x_1,v_{1,1},\dots,v_{m,1})\in \mathcal{H}\times \varOmega_i \times \cdots \times \varOmega_m.$
				\item strictly positive real numbers $\tau, \sigma_i$, $i=1,2,...,m$ be such that $\tau\sum_{i=1}^{m} \sigma_i \|L_i\|^2 < 4 $.
				\item sequences $e_{n} \in (0,1),\theta_n \in (0,2]$
			\end{enumerate}	
		}
		\SetKwBlock{For}{For}{}
		\SetKwProg{}{}{}{}
		\For($n\in \mathbb{N}$; ){
			$p_{1,n}=J_{\tau A} (e_n x_n-\frac{\tau}{2}e_n \sum_{i=1}^{m}L_i^{*}v_{i,n} ) $\\
			$w_{1,n}=2p_{1,n}-e_n x_n$\\
			
			\For($i=1,\dots, m$; ){
				
				$	p_{2,i,n}=J_{\sigma_i P_i ^{-1}}(e_n v_{i,n} +\frac{\sigma_i}{2}L_i w_{1,n})$\\
				$	w_{2,i,n}=2p_{2,i,n}-e_n v_{i,n} $  \\
			}
			
			$	z_{1,n}=w_{1,n}-\frac{\tau}{2}\sum_{i=1}^{m}L_i^{*}w_{2,i,n}$\\
			$	u_{1,n}=e_n x_n+\theta_{n}(z_{1,n}-p_{1,n})$\\
			
			\For($i=1,\dots, m$; ){
				$z_{2,i,n}=J_{\sigma_i Q_i ^{-1}}(w_{2,i,n} + \frac{\sigma_i}{2}L_i (2z_{1,n}-w_{1,n}))$\\
				$u_{2,i,n}=e_n v_{i,n} +\theta_{n}(z_{2,i,n}-p_{2,i,n})$\\
				
			}
			
			$	q_{1,n}=J_{\tau A}(u_{1,n}-\frac{\tau}{2}\sum_{i=1}^{m}L_i^{*}(u_{2,i,n}))$\\
			$	s_{1,n}=2q_{1,n}-u_{1,n}$\\
			
			\For($i=1,\dots, m$; ){
				$	q_{2,i,n}=J_{\sigma_i P_i ^{-1}}(u_{2,i,n}+\frac{\sigma_i}{2}L_i s_{1,n})$\\
				$	s_{2,i,n}=2q_{2,i,n}-u_{2,i,n}$\\
			}
			
			$d_{1,n}=s_{1,n}-\frac{\tau}{2}\sum_{i=1}^{m}L_i^{*}(s_{2,i,n})$\\
			$x_{n+1}=2d_{1,n}-s_{1,n}$\\
			\For($i=1,\dots, m$; ){
				$d_{2,i,n}=J_{\sigma_i Q_i ^{-1}}(s_{2,i,n}+\frac{\sigma_i}{2}L_i(x_{n+1}))$\\
				$v_{2,i,n}=2d_{2,i,n}-s_{2,i,n}$
				
			}
			
		}

		\KwOut{$(x_{n+1},v_{1,n+1},\dots,v_{m,n+1})$}		
		\caption{To optimize the complexly structured monotone inclusion Problem \ref{P4P6.2} }
		\label{P3Al4.1}
	\end{algorithm}	
	
	where sequences $\{\theta_{n}\}$ and $\{e_{n}\}$ satisfy the condition (\textbf{i}) in Theorem \ref{P4T3.1} and the conditions:
	
	$$0< \underline{\Theta} \leq \frac{\theta_{n}}{2} \leq \overline{\Theta } < 1 \textrm{ and } \sum_{n=1}^{\infty} |\theta_{n}-\theta_{n-1}| < \infty.$$
	
	Then there exist an element $(\bar{x}, \bar{v}_1,\dots,\bar{v}_m)\in \mathcal{H}\times \varOmega_1 \times \cdots \times \varOmega_m$ such that following statements are true:
	\begin{enumerate}
		\item[(a)] Denote\\
		$\bar{p}_1=J_{\tau A}\left( \bar{x}-\frac{\tau}{2}\sum_{i=1}^{m}L_i ^* \bar{v}_i\right) $\\
		$\bar{p}_{2,i}=J_{\sigma_i P_i ^{-1}}\left( \bar{v}_i +\frac{\sigma_i}{2}L_i(2\bar{p}_1-\bar{x})\right) , i=1,\dots, m.$
		Then the point $(\bar{p}_1, \bar{p}_{2,1},\dots, \bar{p}_{2,m})\in \mathcal{H} \times \varOmega_1 \times\cdots \times \varOmega_m $ is a primal-dual solution of Problem \ref{P4P6.2}.
		\item[(b)] $\{(x_n,v_{1,n}, \dots, v_{m,n})\}$ converges strongly to $(\bar{x},\bar{v}_1,\dots, \bar{v}_m)$.
		\item[(c)] $\{(p_{1,n}, p_{2,1,n},\dots, p_{2,m,n})\}$ and $\{(z_{1,n}, z_{2,1,n},\dots, z_{2,m,n})\}$ converge strongly to $(\bar{p}_1, \bar{p}_{2,1},\dots, \bar{p}_{2,m})$.
	\end{enumerate}
\end{Theorem}

\begin{proof}
	Consider the real Hilbert space $\mathcal{K}$ and operators
	\begin{enumerate}
		\item[(i)] $\phi: \mathcal{K} \to 2^\mathcal{K}$, defined by $(x,u_1,\dots,u_m)\mapsto (Ax,P_1^{-1}u_1,\dots,P_m^{-1}u_m),$
		
		\item[(ii)] $\xi:\mathcal{K} \to \mathcal{K}$, defined by $(x,u_1,\dots,u_m)\mapsto \left(\sum_{i=1}^{m}L_i ^*u_i,-L_1 x,\dots, -L_m x\right),$
		\item[(iii)] $\psi: \mathcal{K} \to 2^\mathcal{K}$, defined by
		$\psi (x,u_1,\dots,u_m)= \left(0,Q_1^{-1}u_1,\dots, Q_m^{-1}u_m\right).$
	\end{enumerate}
	
	We can observe the following:
	\begin{enumerate}
		\item[(i)] operator $\frac{1}{2} \xi + \psi \text{ and } \frac{1}{2}\xi+\phi $ are maximally monotone as dom $\xi=\mathcal{K}$ (\cite[Corollary 24.4(i)]{Bauschke2011}),
		\item[(ii)]  condition (\ref{P4P6.10}) implies $\operatorname{Zer}(\phi+\xi+\psi) \neq \emptyset$,
		\item[(iii)] \label{P4R6.2}every point in $\operatorname{Zer}(\phi+\xi+\psi)$ solves Problem \ref{P4P6.2}.
	\end{enumerate}
	Define the linear continuous operator $\textbf{W}:\mathcal{K}\to \mathcal{K}$, defined by
	$$\textbf{W}(x,u_1,\dots,u_m )= \left( \frac{x}{\tau}-\frac{1}{2}\sum_{i=1}^{m}L_i ^*u_i,\frac{u_1}{\sigma_1}-\frac{1}{2}L_1x, \dots, \frac{u_m}{\sigma_m}-\frac{1}{2}L_m x \right) $$  which is self-adjoint. Consider $$\rho=\left( 1-\frac{1}{2}\sqrt{\tau \sum_{i=1}^{m}\sigma_i\|L_i\|^2} \right) \min\left\lbrace \frac{1}{\tau}, \frac{1}{\sigma_1},\dots, \frac{1}{\sigma_m} \right\rbrace >0.$$ The operator $\textbf{W}$ is $\rho$- strongly positive in $\mathcal{K}_\textbf{W}$ (\cite{BC2013}) and satisfies the following inequality
	$$\langle x, \textbf{W}x \rangle _\mathcal{K} \geq \rho \|x\|_\mathcal{K} ^2  \ \ \forall x \in \mathcal{K}.$$ Thus the inverse of $\textbf{W}$ exists and satisfies $\|\textbf{W}^{-1}\| \leq \frac{1}{\rho}.$ Consider the sequences
	\begin{eqnarray}{
			\forall n\in \mathbb{N}	\left\{
			\begin{array}{lc@{}c@{}r}
				\text{\textbf{x}}_n = (x_n, v_{1,n}, \dots,v_{m,n})\\
				\text{\textbf{y}}_n = (p_{1,n},p_{2,1,n},\dots,p_{2,m,n})  \\
				\text{\textbf{z}}_n = (z_{1,n},z_{2,1,n},\dots,z_{2,m,n})\\
				\text{\textbf{u}}_n = (u_{1,n},u_{2,1,n},\dots,u_{2,m,n})\\
				\text{\textbf{q}}_n = (q_{1,n},q_{2,1,n},\dots,q_{2,m,n}) \\
				\text{\textbf{s}}_n = (s_{1,n},s_{2,1,n},\dots,s_{2,m,n}) \\
				\text{\textbf{d}}_n = (d_{1,n},d_{2,1,n},\dots,d_{2,m,n}).
			\end{array}\right.}
	\end{eqnarray}
	Using the definition of operators $\phi, \xi,\psi$ and $\textbf{W}$, Algorithm \ref{P3Al4.1} can be written equivalently as
	
	\begin{eqnarray}{
			\forall n\in \mathbb{N}	\left\{
			\begin{array}{lc@{}c@{}r}
				\textbf{W}(e_n \textbf{x}_n-\textbf{y}_n ) \in (\frac{1}{2}\xi+\phi)\textbf{y}_n\\
				\textbf{W}(2\textbf{y}_n- e_n\textbf{x}_n-\textbf{z}_n ) \in (\frac{1}{2}\xi+\psi)\textbf{z}_n\\
				\textbf{u}_n= e_n \textbf{x}_n+\theta_{n}(\textbf{z}_n -\textbf{y}_n)\\
				\textbf{W}(\textbf{u}_n -\textbf{q}_n)\in (\frac{1}{2}\xi+\phi)\textbf{q}_n\\
				\textbf{W}(\textbf{s}_n -\textbf{d}_n)\in (\frac{1}{2}\xi+\psi)(\textbf{d}_n)\\
				\textbf{x}_{n+1}=2\textbf{d}_n-\textbf{s}_n,
			\end{array}\right.}
	\end{eqnarray}
	which is further equivalent to
	\begin{eqnarray}{
			\forall n\in \mathbb{N}	\left\{
			\begin{array}{lc@{}c@{}r}
				\textbf{y}_n =(Id+\textbf{W}^{-1}(\frac{1}{2}\xi+\phi))^{-1}(e_n\textbf{x}_n)  \\
				{\textbf{z}_n=(Id+\textbf{W}^{-1}(\frac{1}{2}\xi+\psi))^{-1}(2\textbf{y}_n-e_n\textbf{x}_n) }\\
				{\textbf{u}_n= e_n\textbf{x}_n+\theta_{n}(\textbf{z}_n -\textbf{y}_n)}\\
				\textbf{q}_n= (Id+\textbf{W}^{-1}(\frac{1}{2}\xi+\phi))^{-1}(\textbf{u}_n)\\
				\textbf{s}_n = 2\textbf{q}_n - \textbf{u}_n\\
				\textbf{d}_n=(Id+\textbf{W}^{-1}(\frac{1}{2}\xi+\psi))^{-1}(\textbf{s}_n)\\
				\textbf{x}_{n+1}=2\textbf{d}_n-\textbf{s}_n.
			\end{array}\right.}
	\end{eqnarray}
	Now, consider the real Hilbert space $\mathcal{K}_{\text{\textbf{W}}}=\mathcal{H}\times \varOmega_1\times \cdots \times \varOmega_m$ with inner product and norm defined as \\
	$\langle \textbf{x}, \textbf{y} \rangle_{\mathcal{K}_\textbf{W}} = \langle \textbf{x}, \textbf{W}\textbf{y} \rangle $ and $\|\textbf{x}\|_{\mathcal{K}_\textbf{W}}= \sqrt{\langle \textbf{x}, \textbf{W}\textbf{x} \rangle_{\mathcal{K}}}$ respectively.\\
	Now, define the operators $\textbf{A} \equiv \textbf{W}^{-1}(\frac{1}{2}\xi+\psi)$ and $\textbf{B} \equiv \textbf{W}^{-1}(\frac{1}{2}\xi+\phi),$ which are maximally monotone on $\mathcal{K}_{\textbf{W}}$ as $\frac{1}{2}\xi+\phi$ and $\frac{1}{2}\xi+\psi$ are maximally monotone on $\mathcal{K}.$ The Algorithm \ref{P3Al4.1} can be written in the form of Douglas-Rachford algorithm as
	\begin{eqnarray}{
			\forall n\in \mathbb{N}\left\{
			\begin{array}{lc@{}c@{}r}
				\textbf{y}_n = \textbf{J}_{\textbf{B}}(e_n \textbf{x}_n)\\
				\textbf{u}_n = e_n\textbf{x}_n+\theta_{n}\left( \textbf{J}_{\textbf{A}}(2\textbf{y}_n - e_n \textbf{x}_n)-\textbf{y}_n\right) \\
				\textbf{x}_{n+1}=(2\textbf{J}_{\textbf{A}}- Id) (2\textbf{J}_{\textbf{B}}- Id)\textbf{z}_n,
			\end{array}\right.}
	\end{eqnarray}
	which is of the form Algorithm (\ref{e5.1}) for $\gamma =1.$ From assumption (\ref{P4P6.10}), we have
	$$\operatorname{Zer}(\textbf{A} +\textbf{B})=\operatorname{Zer}(\textbf{W}^{-1}(\textbf{M}+\textbf{S}+\textbf{Q}))=\operatorname{Zer}(\textbf{M}+\textbf{S}+\textbf{Q}).$$
	Applying Theorem \ref{T5.1}, we can find $\bar{\textbf{x}} \in \operatorname{Fix}(R_{\textbf{A}}R_{\textbf{B}})$ such that $\textbf{J}_{\text{\textbf{B}}}{\bar{\textbf{x}}} \in \operatorname{Zer}(\textbf{A}+\textbf{B}).$
\end{proof}

At the end of this section, we study iterative technique to solve  following convex optimization problem:
\begin{Problem}\label{P4P5.33}
	Let $f \in \Gamma(\mathcal{H})$. Consider $\varOmega_i$ are real Hilbert spaces, $g_i,l_i \in \Gamma(\varOmega_i)$ and $L_i : \mathcal{H} \to \varOmega_i$ are linear continuous operators, $i=1,\dots,m$. The optimization problem is given by
	\begin{equation}
		\inf_{x\in \mathcal{H}}\left[f(x)+\sum_{i=1}^{m}(g_i\square l_i)(L_i x)\right]
	\end{equation}
	with conjugate-dual problem is given by
	\begin{equation}
		\sup_{v_i \in \varOmega, i=1,2,\dots,m}\left\lbrace -f^{*}\left( -\sum_{i=1}^{m}L_i^* v_i\right) -\sum_{i=1}^{m}(g_i^*(v_i)+l_i^*(v_i))\right\rbrace.
	\end{equation}
\end{Problem}

Consider stricly positive integers $\tau, \sigma_i, i=1,\dots,m$ and
initial point $(x_1,v_{1,1},\dots,v_{m,1})\in \mathcal{H}\times \varOmega_i \times \cdots \times \varOmega_m.$ The primal-dual algorithm to solve Problem \ref{P4P5.33} is given by

\begin{algorithm}[H]
	\SetAlgoLined
	
	\KwIn{\begin{enumerate}
			\item initial points $(x_1,v_{1,1},\dots,v_{m,1})\in \mathcal{H}\times \varOmega_i \times \cdots \times \varOmega_m.$
			\item \label{P4As5.32}Positive real numbers $\tau, \sigma_i$, $i=1,2,...,m$ be such that $\tau\sum_{i=1}^{m} \sigma_i \|L_i\|^2 < 4 $.
			\item The sequences $\{\theta_n\}, \{e_n\}$ satisfying the assumptions in  Theorem \ref{P4T5.11}.
		\end{enumerate}	
	}
	\SetKwBlock{For}{For}{}
	\SetKwProg{}{}{}{}
	\For($n\in \mathbb{N}$; ){
		$p_{1,n}=prox_{\tau f} (e_n x_n-\frac{\tau}{2}e_n \sum_{i=1}^{m}L_i^{*}v_{i,n} )$\\
		$w_{1,n}=2p_{1,n}-e_n x_n$\\
		
		\For($i=1,\dots, m$; ){
			
			$	p_{2,i,n}=prox_{\sigma_i g_i ^*}(e_n v_{i,n} +\frac{\sigma_i}{2}L_i w_{1,n})$\\
			$	w_{2,i,n}=2p_{2,i,n}-e_n v_{i,n}  $ \\
		}
		
		$z_{1,n}=w_{1,n}-\frac{\tau}{2}\sum_{i=1}^{m}L_i^{*}w_{2,i,n}$\\
		$u_{1,n}=e_n x_n+\theta_{n}(z_{1,n}-p_{1,n})$\\
		
		\For($i=1,\dots, m$; ){
			$z_{2,i,n}=prox_{\sigma_i l_i ^ *}(w_{2,i,n} + \frac{\sigma_i}{2}L_i (2z_{1,n}-w_{1,n}))$\\
			$	u_{2,i,n}=e_n v_{i,n} +\theta_{n}(z_{2,i,n}-p_{2,i,n})$\\
			
		}
		
		$q_{1,n}=prox_{\tau f}(u_{1,n}-\frac{\tau}{2}\sum_{i=1}^{m}L_i^{*}(u_{2,i,n}))$\\
		$s_{1,n}=2q_{1,n}-u_{1,n}$\\
		
		\For($i=1,\dots, m$; ){
			$	q_{2,i,n}=prox_{\sigma_i g_i ^*}(u_{2,i,n}+\frac{\sigma_i}{2}L_i s_{1,n})$\\
			$	s_{2,i,n}=2q_{2,i,n}-u_{2,i,n}$\\
		}
		
		$	t_{1,n}=s_{1,n}-\frac{\tau}{2}\sum_{i=1}^{m}L_i^{*}(s_{2,i,n})$\\
		$	x_{n+1}=2t_{1,n}-s_{1,n}$\\
		\For($i=1,\dots, m$; ){
			$t_{2,i,n}=prox_{\sigma_i l_i ^*}(s_{2,i,n}+\frac{\sigma_i}{2}L_i(x_{n+1}))$\\
			$	v_{2,i,n}=2t_{2,i,n}-s_{2,i,n}$
			
		}
		
	}

	\KwOut{$(x_{n+1},v_{1,n+1},\dots,v_{m,n+1})$}		
	\caption{To optimize the complexly structured monotone inclusion Problem \ref{P4P5.33} }\label{P3Al4.3}
\end{algorithm}	
where $\{\theta_{n}\} \text{ and } \{e_n\}$  are real sequences.

\begin{Corollary}
	In addition to assumptions in Problem \ref{P4P5.33}, consider
	\begin{equation}\label{e6.}
		0\in ran\left( \partial f +\sum_{i=1}^{m}L_i ^* \circ (\partial g_i \square \partial l_i)\circ L_i \right) .
	\end{equation}	
	Then, there exists an element $(\bar{x}, \bar{v}_1,\dots,\bar{v}_m)\in \mathcal{H}\times \varOmega_1 \times \cdots \times \varOmega_m$ such that following statements are true:
	\begin{enumerate}
		\item[(a)] Denote\\
		$\bar{p}_1=prox_{\tau f}\left( \bar{x}-\frac{\tau}{2}\sum_{i=1}^{m}L_i ^* \bar{v}_i\right) $\\
		$\bar{p}_{2,i}=prox_{\sigma_i g^*_i}\left( \bar{v}_i +\frac{\sigma_i}{2}L_i(2\bar{p}_i-\bar{x})\right) , i=1,\dots, m.$
		Then the point $(\bar{p}_1, \bar{p}_{2,1},\dots, \bar{p}_{2,m})\in \mathcal{H} \times \varOmega_1 \times\cdots \times \varOmega_m $ is a primal-dual solution of Problem \ref{P4P5.33}.
		\item[(b)] $\{(x_n,v_{1,n}, \dots, v_{m,n})\}$ converges strongly to $(\bar{x},\bar{v}_1,\dots, \bar{v}_m)$.
		\item[(c)] $\{(p_{1,n}, p_{2,1,n},\dots, p_{2,m,n})\}$ and $\{(z_{1,n}, z_{2,1,n},\dots, z_{2,m,n})\}$ converge strongly to $(\bar{p}_1, \bar{p}_{2,1},\dots, \bar{p}_{2,m})$.
	\end{enumerate}
\end{Corollary}

\section{Numerical Experiment}
In this section, we make an experimental setup to solve the wavelet based image deblurring problem. In image deblurring, we develop mathematical methods to recover the original, sharp image from blurred image. The mathematical formulation of the blurring process can be written as linear inverse problem,
\begin{equation} \label{p4n1}
	\text{ find } x \in \mathbb{R}^d \text{ such that }	Ax=b+w
\end{equation}
where $A\in \mathbb{R}^{m\times d}$ is a blurring operator, $b\in \mathbb{R}^m$ is blurred image and $w$ is an unknown noise.
A classical approach to solve Problem (\ref{p4n1}) is to minimize the least-square term
$\|Ax-b\|^2$. In the deblurring case, the problem is ill-conditioned as the norm solution has a huge norm. To remove the difficulty, the ill-conditioned problem is replaced by a nearly well-conditioned problem. In wavelet domain, most images are sparse in nature, that's why we choose $l_1$ regularization. For $l_1$ regularization, the image processing problem becomes
\begin{equation}\label{p4n2}
	\min_{x\in\mathbb{R}^2}F(x)=	\|Ax-b\|^2 +\lambda \|x\|_1,
\end{equation}
where $\lambda$ is a sparsity controlling parameter and  provides a tradeoff between fidelity to the measurements and noise sensitivity. The $l_1$ regularization produces sparse images having sharp edges since it is less sensitive to outliers.
Using subdifferential characterization of the minimum of a convex function, a point $x^*$ minimizes $F(x)$ if and only if
\begin{equation}\label{p4n3}
	0\in A^T(Ax^*-b)+\partial \lambda \|x^*\|_1.\nonumber
\end{equation}
Thus we can apply forward-backward Algorithm (\ref{e4.1}) with $A_1=A_2=A^T(Ax^*-b)$ and $B_1=B_2=\partial \lambda \|x^*\|_1$ to solve the deblurring problem (\ref{p4n2}).

For numerical experiment purposes, we have chosen images from publicly available domain and assumed reflexive (Neumann) boundary conditions. We blurred the images using gaussian blur of size $9 \times 9$ and standard deviation 4. We have compared the algorithm (\ref{e4.1}) with \cite[Algorithm 8]{Bot2019}. The operator $A=RW$, where $W$ is the three stage haar wavelet transform and $R$ is the blur operator. The original and corresponding blurred images were shown in Figure \ref{Fig6.5}.
\begin{figure}[htb!]
	\begin{subfigure}[b]{0.48\textwidth}
		\includegraphics[width=\linewidth]{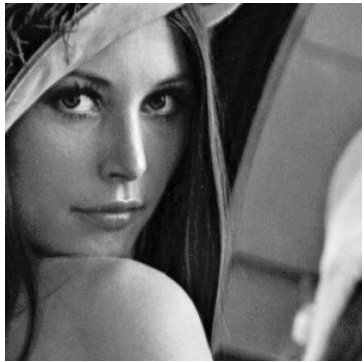}
		\caption{Original.}
		\label{Fig6.1}
	\end{subfigure}
	\hfill
	\begin{subfigure}[b]{0.48\textwidth}
		\includegraphics[width=\linewidth]{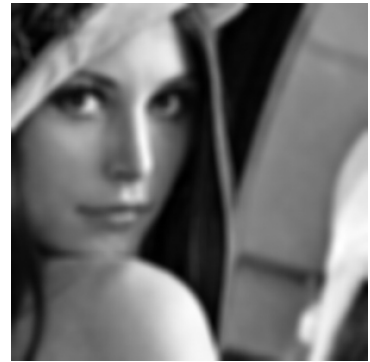}
		\caption{Blurred}
		\label{Fig6.2}
	\end{subfigure}
	\hfill
	\begin{subfigure}[b]{0.48\textwidth}
		\includegraphics[width=\linewidth]{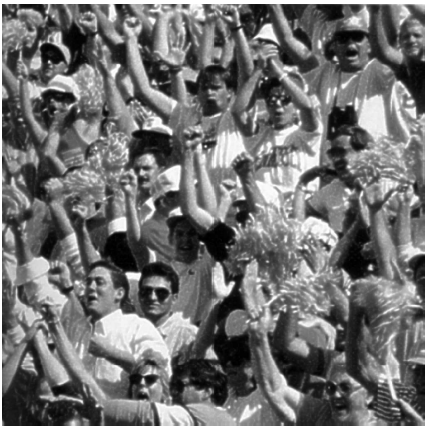}
		\caption{Original}
		\label{Fig6.3}
	\end{subfigure}
	\hfill
	\begin{subfigure}[b]{0.48\textwidth}
		\includegraphics[width=\linewidth]{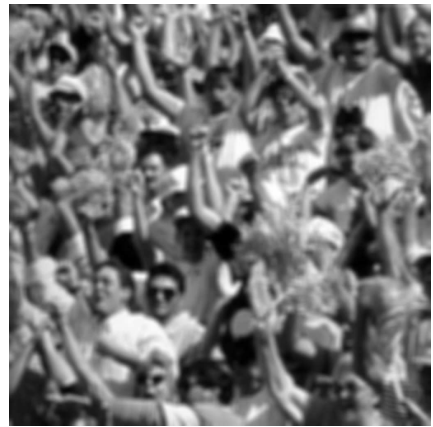}
		\caption{Blurred}
		\label{Fig6.4}
	\end{subfigure}
	\caption{The original and blurred images of Lenna and crowd. Peak signal to noise ratio (PSNR) of image (b) is 28.3111 and of image (d) is 21.3171}
	\label{Fig6.5}
\end{figure}
The regularization parameter was chosen to be $\lambda=2\times10^{-5}$, and the initial image was the blurred image. The objective function value is denoted by $F(x^*)$ and function value at $n^{\text{th}}$ iteration is denoted by $F(x_n)$. Sequences $\{e_n\}$ and $\{\theta_n\}$ are chosen as $\{1-\frac{1}{n+1}\}$ and $\{0.9\}$ respectively. The images recovered by the algorithms for 1000 iterations are shown in the figure.
The graphical representation of convergence of $F(x_n) - F(x^*)$ is depicted in Figure \ref{Fig6.8}.
\begin{figure}[htb!]
	\begin{subfigure}[b]{0.48\textwidth}
		\includegraphics[width=\linewidth]{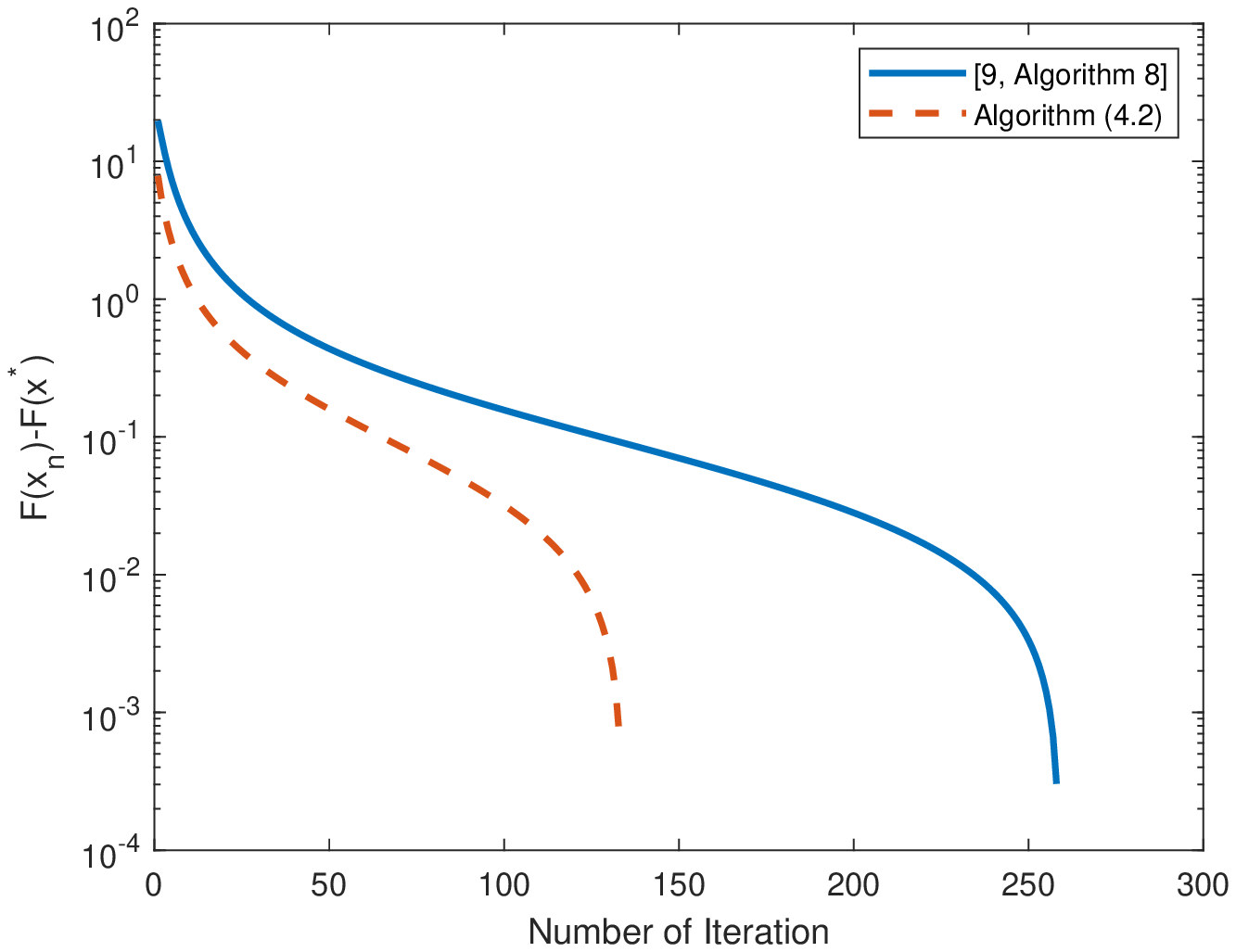}
		\caption{Lenna.}
		\label{Fig6.6}
	\end{subfigure}
	\hfill
	\begin{subfigure}[b]{0.48\textwidth}
		\includegraphics[width=\linewidth]{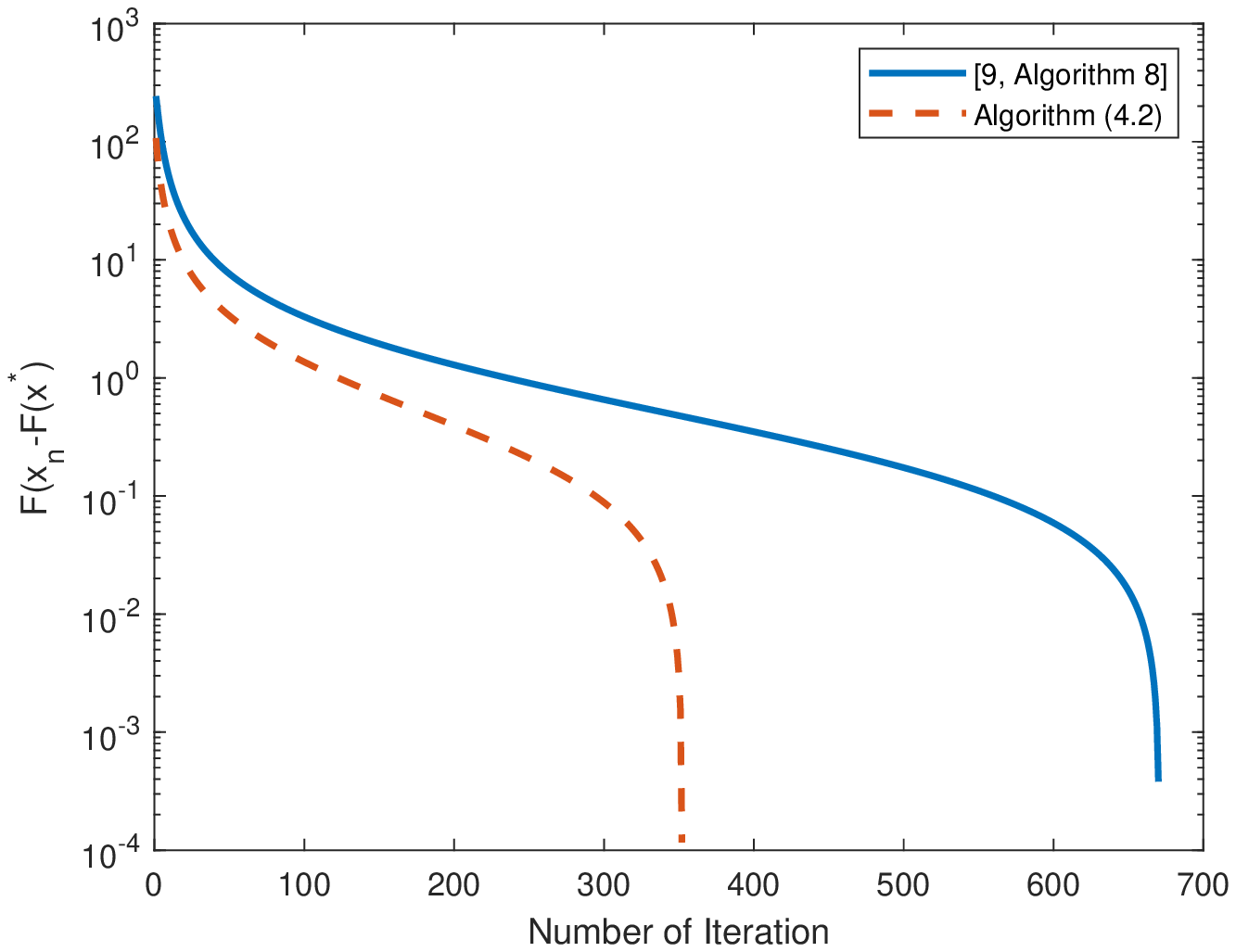}
		\caption{Crowd}
		\label{Fig6.7}
	\end{subfigure}
	\caption{The variation of $F(x_n)-F(x^*)$ with respect to number of iteration for different images.}
	\label{Fig6.8}
\end{figure}
For deblurring methods, lower the value of $F(x_n)-F(x^*)$ higher the quality of recovered images.
\begin{figure}[htb!]
	\begin{subfigure}[b]{0.48\textwidth}
		\includegraphics[width=\linewidth]{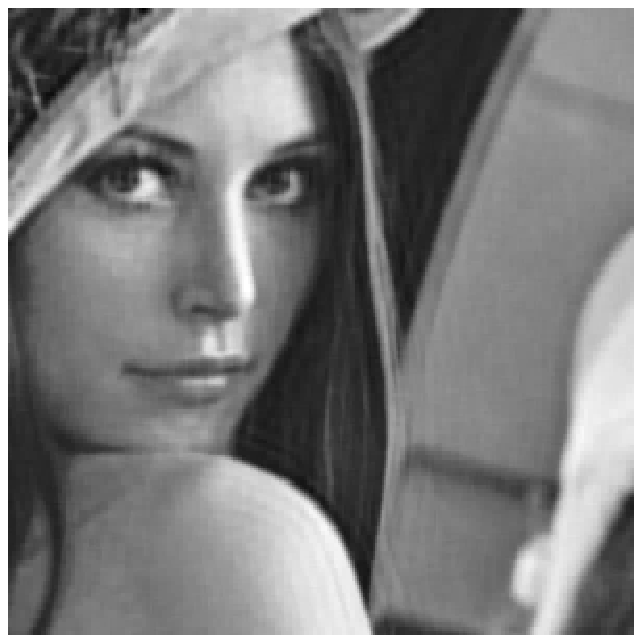}
		\caption{ Algorithm (\ref{e4.1}). PSNR = 34.8448.}
		\label{Fig6.9}
	\end{subfigure}
	\hfill
	\begin{subfigure}[b]{0.48\textwidth}
		\includegraphics[width=\linewidth]{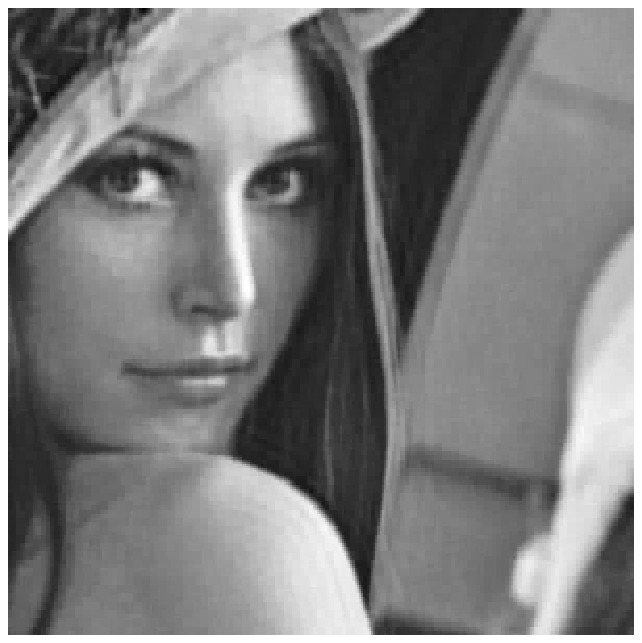}
		\caption{\cite[Algorithm 8]{Bot2019}. PSNR = 34.5123} 
		\label{Fig6.10}
	\end{subfigure}\hfill
	\begin{subfigure}[b]{0.48\textwidth}
		\includegraphics[width=\linewidth]{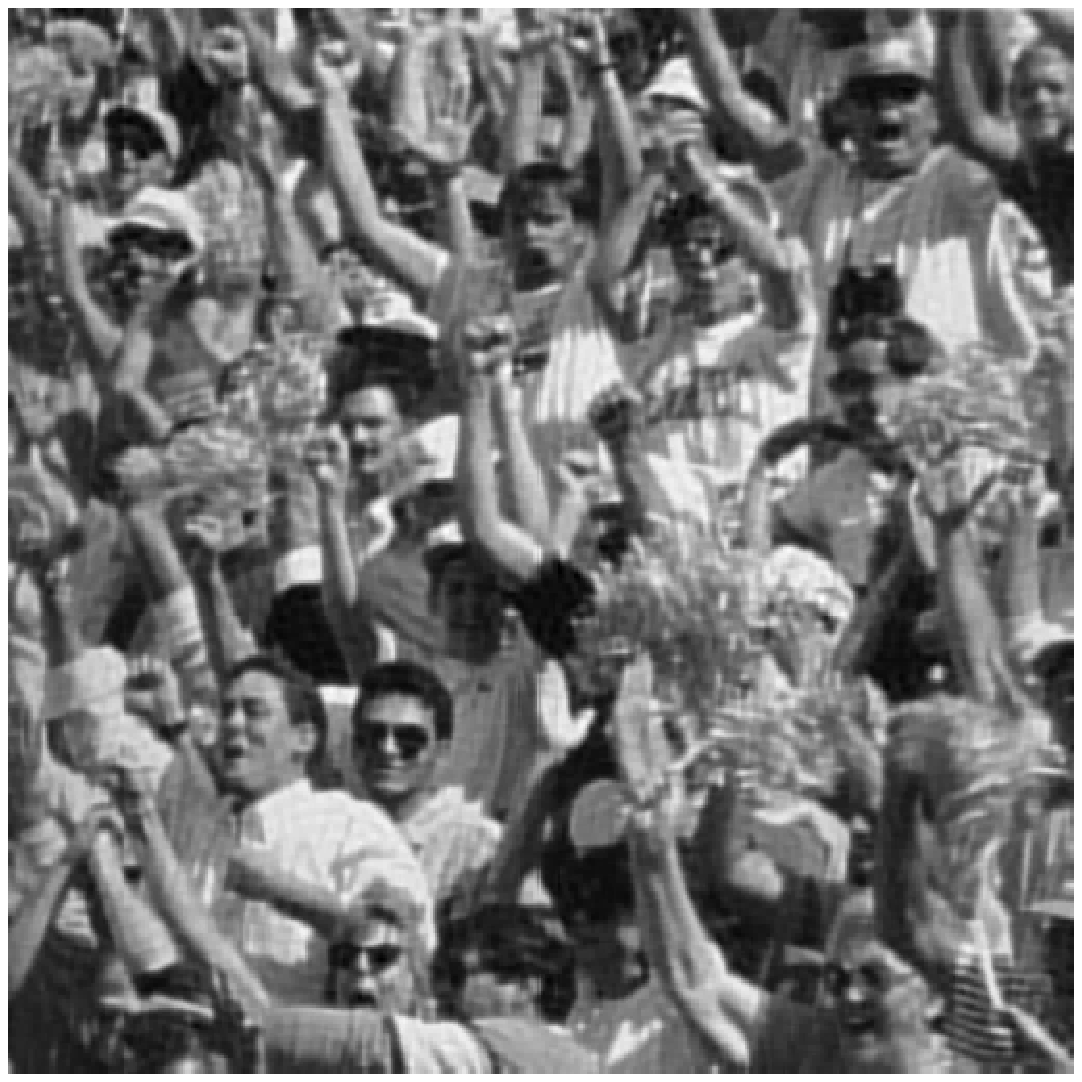}
		\caption{ Algorithm (\ref{e4.1}). PSNR = 29.3581}
		\label{Fig6.11}
	\end{subfigure}
	\hfill
	\begin{subfigure}[b]{0.48\textwidth}
		\includegraphics[width=\linewidth]{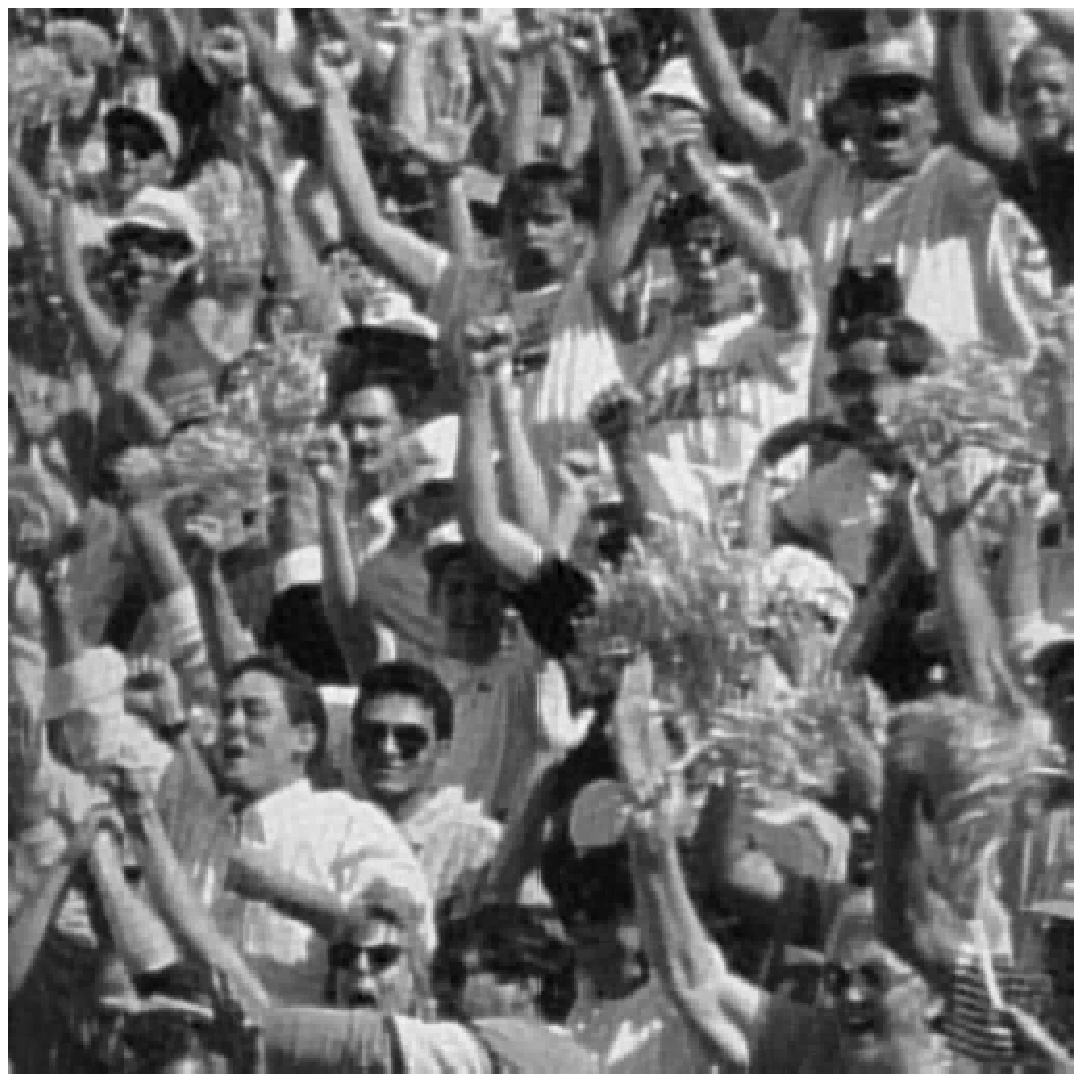}
		\caption{\cite[Algorithm 8]{Bot2019}. PSNR = 28.5764}
		\label{Fig6.12}
	\end{subfigure}
	\caption{The recovered images using different algorithms for 1000 iterations.}
	\label{Fig6.13}
\end{figure}

\begin{figure}[htb!]
	\begin{subfigure}[b]{0.48\textwidth}
		\includegraphics[width=\linewidth]{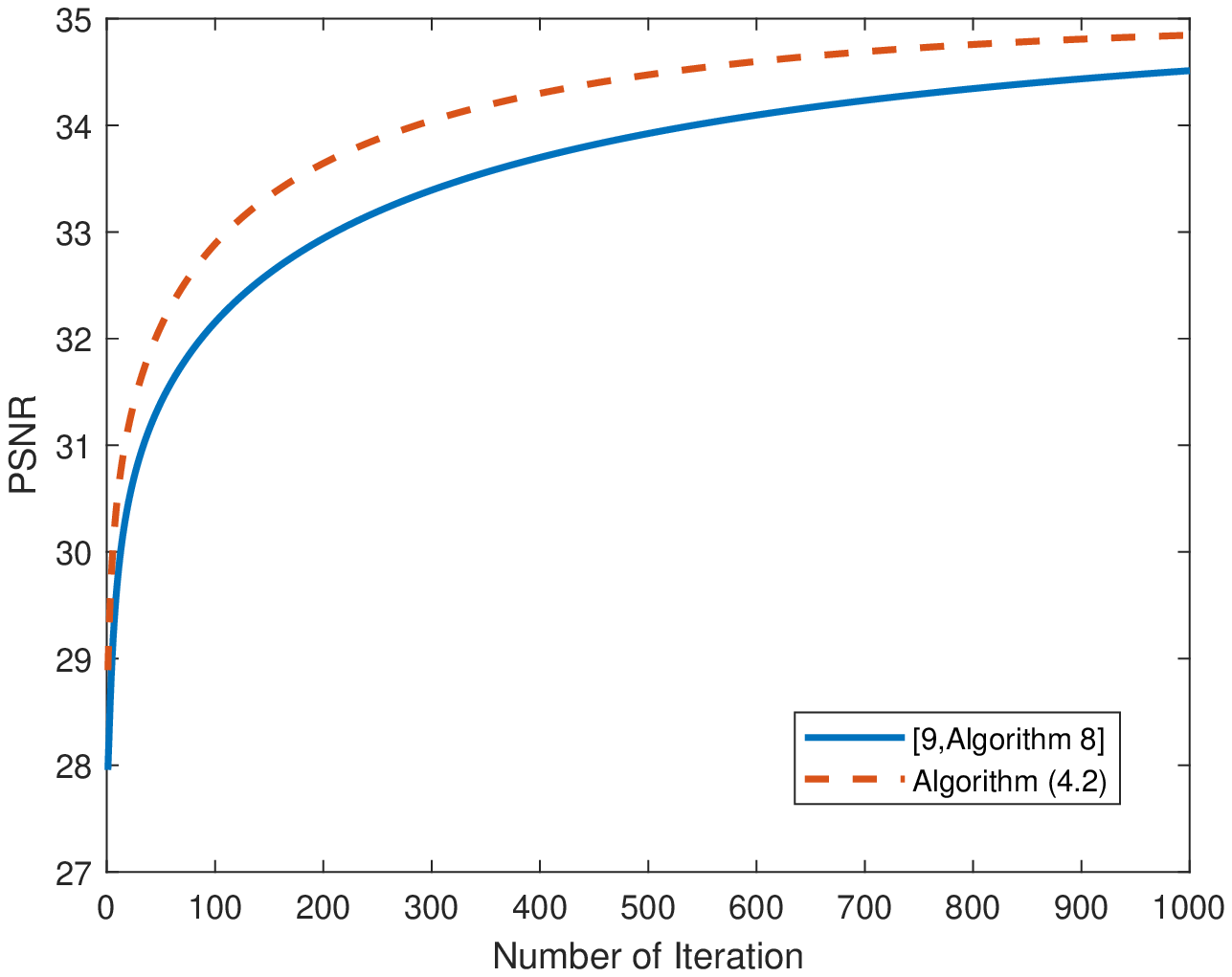}
		\caption{Lenna.}
		\label{Fig6.16}
	\end{subfigure}
	\hfill
	\begin{subfigure}[b]{0.48\textwidth}
		\includegraphics[width=\linewidth]{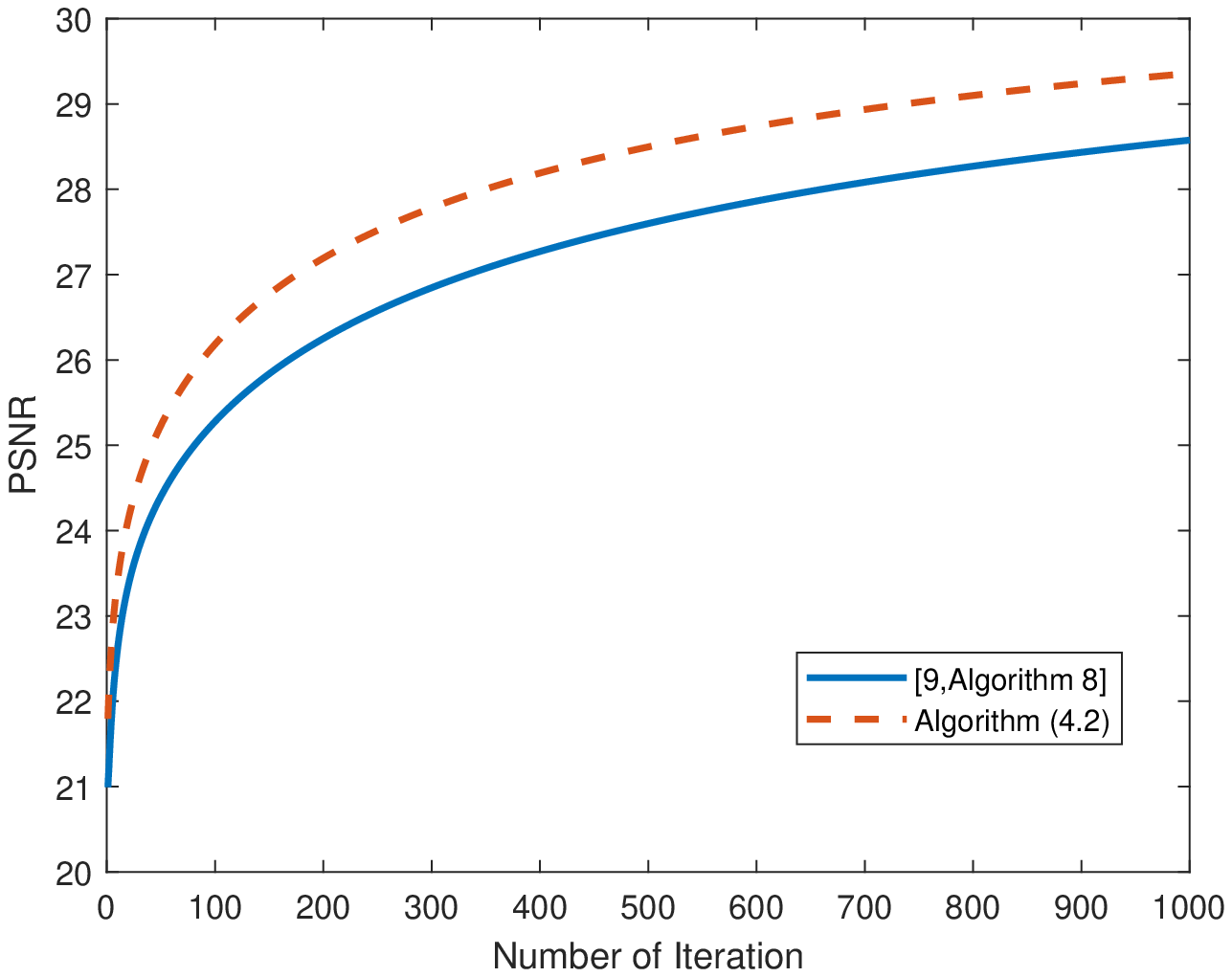}
		\caption{Crowd}
		\label{Fig6.17}
	\end{subfigure}
	\caption{The variation of PSNR value of recovered image at each iteration with original image as a reference image.}
	\label{Fig6.18}
\end{figure}
It can be observed from Figures \ref{Fig6.8} and \ref{Fig6.13}  that proposed Algorithm (\ref{e4.1}) outperforms over \cite[Algorithm 8]{Bot2019}. This can also be confirmed from the peak signal to noise ratio (PSNR) value of the recovered images and the fact that higher the PSNR value, better the image quality. The variation of PSNR value of recovered image at each iteration with original image as a reference is also plotted in Figure \ref{Fig6.18}. 
%


\section*{Acknowledgements}
The first author acknowledges the Indian Institute of Technology (BHU), Varanasi for Senior Research Fellowship in terms of teaching assistantship. The third author is thankful to University Grant Commission for the Senior Research Fellowship.


\begin{thebibliography}\\
	\bibitem{Agarwal2007}
	Agarwal, R. P., O Regan, D., and Sahu, D. R. (2007). Iterative construction of fixed points of nearly asymptotically nonexpansive mappings. Journal of Nonlinear and convex Analysis, 8(1), 61.	
	\bibitem{Artacho2013}
	Artacho, F. J. A., and Borwein, J. M. (2013). Global convergence of a non-convex Douglas–Rachford iteration. Journal of Global Optimization, 57(3), 753-769.
	
	\bibitem{Attouch1996}
	Attouch, H., and Théra, M. (1996). A general duality principle for the sum of two operators. Journal of Convex Analysis, 3, 1-24.
	
	\bibitem{Barty2007}
	Barty, K., Roy, J. S., and Strugarek, C. (2007). Hilbert-valued perturbed subgradient algorithms. Mathematics of Operations Research, 32(3), 551-562.
	
	\bibitem{Bauschke2001}
	Bauschke, H. H., and Combettes, P. L. (2001). A weak-to-strong convergence principle for Fejér-monotone methods in Hilbert spaces. Mathematics of operations research, 26(2), 248-264..
	
	
	\bibitem{Bauschke2011}
	Bauschke, H. H., and Combettes, P. L. (2011). Convex analysis and monotone operator theory in Hilbert spaces (Vol. 408). New York: Springer.
	
	\bibitem{Bennar2007}
	Bennar, A., and Monnez, J. M. (2007). Almost sure convergence of a stochastic approximation process in a convex set. International Journal of Apllied Mathematics, 20(5), 713-722.
	
	\bibitem{Bot2013}
	Bo{\c{t}}, R. I., Csetnek, E. R., and Heinrich, A. (2013). A primal-dual splitting algorithm for finding zeros of sums of maximal monotone operators. SIAM Journal on Optimization, 23(4), 2011-2036.
	
	\bibitem{Bot2019}
	Bo{\c{t}}, R. I., Csetnek, E. R., and Meier, D. (2019). Inducing strong convergence into the asymptotic behaviour of proximal splitting algorithms in Hilbert spaces. Optimization Methods and Software, 34(3), 489-514.
	
	
	\bibitem{Bota2012}
	Bo{\c{t}}, R. I., and Hendrich, C. (2014). Convergence analysis for a primal-dual monotone+ skew splitting algorithm with applications to total variation minimization. Journal of mathematical imaging and vision, 49(3), 551-568.
	
	\bibitem{Bot2021}
Bo{\c{t}}, R. I., and Nguyen, D. K. (2021). Factorization of completely positive matrices using iterative projected gradient steps. Numerical Linear Algebra with Applications, e2391.
	
	
	\bibitem{Brice2011}
	Briceno-Arias, L. M., and Combettes, P. L. (2011). A monotone+ skew splitting model for composite monotone inclusions in duality. SIAM Journal on Optimization, 21(4), 1230-1250.
	
	\bibitem{Briceno2013}
	Briceno-Arias, L. M., and Combettes, P. L. (2013). Monotone operator methods for Nash equilibria in non-potential games. In Computational and analytical mathematics (pp. 143-159). Springer, New York, NY.
	
	\bibitem{Butnariu2008}
	Butnariu, D., and Kassay, G. (2008). A proximal-projection method for finding zeros of set-valued operators. SIAM Journal on Control and Optimization, 47(4), 2096-2136.
	
	
	\bibitem{Chang2013}
	Chang, S. S., Wang, L., Lee, H. W. J., and Chan, C. K. (2013). Strong and $\bigtriangleup$-convergence for mixed type total asymptotically nonexpansive mappings in CAT (0) spaces. Fixed Point Theory and Applications, 1, 1-16.
	
	\bibitem{Chang2014}
	Chang, S. S., Wang, G., Wang, L., Tang, Y. K., and Ma, Z. L. (2014). $\bigtriangleup$-convergence theorems for multi-valued nonexpansive mappings in hyperbolic spaces. Applied Mathematics and Computation, 249, 535-540.
	
	\bibitem{Chen1994}
	Chen, G. H. G. (1994). Forward-backward splitting techniques: theory and applications (Doctoral dissertation, University of Washington).
	
	\bibitem{Chen1997}
	Chen, G. H., and Rockafellar, R. T. (1997). Convergence rates in forward--backward splitting. SIAM Journal on Optimization, 7(2), 421-444.
	
	\bibitem{Cholamjiak2015}
	Cholamjiak, P., Abdou, A. A., and Cho, Y. J. (2015). Proximal point algorithms involving fixed points of nonexpansive mappings in $\operatorname {CAT}(0) $ spaces. Fixed Point Theory and Applications, 1, 1-13.
	
	\bibitem{Combettes2012}
	Combettes, P. L., and Pesquet, J. C. (2012). Primal-dual splitting algorithm for solving inclusions with mixtures of composite, Lipschitzian, and parallel-sum type monotone operators. Set-Valued and variational analysis, 20(2), 307-330.
	
	
	\bibitem{Combettes2005}
	Combettes, P. L., and Wajs, V. R. (2005). Signal recovery by proximal forward-backward splitting. Multiscale Modeling and Simulation, 4(4), 1168-1200.
	
	\bibitem{Daubechies2004}
	Daubechies, I., Defrise, M., and De Mol, C. (2004). An iterative thresholding algorithm for linear inverse problems with a sparsity constraint. Communications on Pure and Applied Mathematics: A Journal Issued by the Courant Institute of Mathematical Sciences, 57(11), 1413-1457.
	
	\bibitem{Davis2015}
	Davis, D. (2015). Convergence rate analysis of the forward-Douglas-Rachford splitting scheme. SIAM Journal on Optimization, 25(3), 1760-1786.
	
	\bibitem{Dixit2019}
	Dixit, A., Sahu, D. R., Singh, A. K., and Som, T. (2020). Application of a new accelerated algorithm to regression problems. Soft Computing, 24(2), 1539-1552.
	
	\bibitem{Douglas1956}
	Douglas, J., and Rachford, H. H. (1956). On the numerical solution of heat conduction problems in two and three space variables. Transactions of the American mathematical Society, 82(2), 421-439.
	
	\bibitem{Gabay1983}
	Fortin, M., and Glowinski, R. (2000). Augmented Lagrangian methods: applications to the numerical solution of boundary-value problems. Elsevier.
	
	\bibitem{Gautam2020}
Gautam, P., Dixit, A., Sahu, D. R., and Som, T. (2020). Application of new strongly convergent iterative methods to split equality problems. Computational and Applied Mathematics, 39, 1-28.
	
	\bibitem{Gautam2021}
	Gautam, P., Sahu, D. R., Dixit, A., and Som, T. (2021). Forward–Backward–Half Forward Dynamical Systems for Monotone Inclusion Problems with Application to v-GNE. Journal of Optimization Theory and Applications, 190(2), 491-523.
	
	
	\bibitem{Guler1991}
	Güler, O. (1991). On the convergence of the proximal point algorithm for convex minimization. SIAM Journal on Control and Optimization, 29(2), 403-419.
	
	\bibitem{Gur2020}
Gur, E., Sabach, S., and Shtern, S. (2020). Alternating minimization based first-order method for the wireless sensor network localization problem. IEEE Transactions on Signal Processing, 68, 6418-6431.
	
	\bibitem{Haugazeau1968}
	Haugazeau, Y. (1968). Sur les inéquations variationnelles et la minimisation de fonctionnelles convexes. These, Universite de Paris.
	
	\bibitem{Iemoto2009}
	Iemoto, S., and Takahashi, W. (2009). Approximating common fixed points of nonexpansive mappings and nonspreading mappings in a Hilbert space. Nonlinear analysis: theory, methods and applications, 71(12), e2082-e2089.
	
	\bibitem{Lehdili1996}
	Lehdili, N., and Moudafi, A. (1996). Combining the proximal algorithm and Tikhonov regularization. Optimization, 37(3), 239-252.
	
	
	
	
	\bibitem{Lions1979}
	Lions, P. L., and Mercier, B. (1979). Splitting algorithms for the sum of two nonlinear operators. SIAM Journal on Numerical Analysis, 16(6), 964-979.
	
	\bibitem{Luke2020}
	Luke, D. R., and Martins, A. L. (2020). Convergence Analysis of the Relaxed Douglas--Rachford Algorithm. SIAM Journal on Optimization, 30(1), 542-584.
	
	\bibitem{Mainge2007}
Maingé, P. E. (2007). Approximation methods for common fixed points of nonexpansive mappings in Hilbert spaces. Journal of Mathematical Analysis and Applications, 325(1), 469-479.
	
	\bibitem{Mann1953}
	Mann, W. R. (1953). Mean value methods in iteration. Proceedings of the American Mathematical Society, 4(3), 506-510.
	
	\bibitem{Martinet1972}
	Martinet, B. (1972). Détermination approchée d’un point fixe d’une application pseudo-contractante. CR Acad. Sci. Paris, 274(2), 163-165.
	
	\bibitem{Mercier1980}
	Mercier, B. (1980). Inéquations variationnelles de la mécanique. Université de Paris-Sud, Département de mathématique.
	
	\bibitem{Mouallif1991}
	Mouallif, K., Nguyen, V. H., and Strodiot, J. J. (1991). A perturbed parallel decomposition method for a class of nonsmooth convex minimization problems. SIAM journal on control and optimization, 29(4), 829-847.
	
	\bibitem{Moudafi1997}
	Moudafi, A., and Théra, M. (1997). Finding a zero of the sum of two maximal monotone operators. Journal of Optimization Theory and Applications, 94(2), 425-448.
	
	\bibitem{Nemirovski2009}
	Nemirovski, A., Juditsky, A., Lan, G., and Shapiro, A. (2009). Robust stochastic approximation approach to stochastic programming. SIAM Journal on optimization, 19(4), 1574-1609.
	
	
	\bibitem{Passty1979}
	Passty, G. B. (1979). Ergodic convergence to a zero of the sum of monotone operators in Hilbert space. Journal of Mathematical Analysis and Applications, 72(2), 383-390.
	
	
	\bibitem{Phan2016}
	Phan, H. M. (2016). Linear convergence of the Douglas–Rachford method for two closed sets. Optimization, 65(2), 369-385.
	
	\bibitem{Rockafella1976}
	Rockafellar, R. T. (1976). Augmented Lagrangians and applications of the proximal point algorithm in convex programming. Mathematics of operations research, 1(2), 97-116.
	
	\bibitem{Rockafellar1976}
	Rockafellar, R. T. (1976). Monotone operators and the proximal point algorithm. SIAM journal on control and optimization, 14(5), 877-898.
	
	\bibitem{Sahu2011}
	Sahu, D. R. (2011). Applications of the S-iteration process to constrained minimization problems and split feasibility problems. Fixed Point Theory, 12(1), 187-204.
	
	\bibitem{Sahu2020}
	Sahu, D. R., Pitea, A., and Verma, M. (2020). A new iteration technique for nonlinear operators as concerns convex programming and feasibility problems. Numerical Algorithms, 83(2), 421-449.
	
	\bibitem{Sahuajit2020}
	Sahu, D. R., Kumar, A., and Kang, S. M. (2021). Proximal point algorithms based on S-iterative technique for nearly asymptotically quasi-nonexpansive mappings and applications. Numerical Algorithms, 86(4), 1561-1590.
	
	\bibitem{Svaiter2011}
	Svaiter, B. F. (2011). On weak convergence of the Douglas–Rachford method. SIAM Journal on Control and Optimization, 49(1), 280-287.
	
	\bibitem{Tikhonov1965}
	Tikhonov, A. N. (1965). Improper problems of optimal planning and stable methods of their solution. In Soviet Math. Doklady (Vol. 6, pp. 1264-1267).
	
	\bibitem{Tikhonov1977}
	Tikhonov, A.N., Arsenin, V.J., (1977). Methods for Solving Ill-Posed Problems. Wiley, New York.
	
	\bibitem{Tikhonov1963}
	Tihonov, A. N. (1963). Solution of incorrectly formulated problems and the regularization method. Soviet Math., 4, 1035-1038.
	
	\bibitem{BC2013}
	Vũ, B. C. (2013). A splitting algorithm for dual monotone inclusions involving cocoercive operators. Advances in Computational Mathematics, 38(3), 667-681.
	
	\bibitem{Xu2002}
	Xu, H. K. (2002). Iterative algorithms for nonlinear operators. Journal of the London Mathematical Society, 66(1), 240-256.
	
	
	
	
	
	
	
	
	
	
\end{thebibliography}
\end{document}